\documentclass[a4paper,10pt]{amsart}
\usepackage[utf8]{inputenc}
\usepackage[T1]{fontenc}

\usepackage{amsmath}
\usepackage{amssymb}
\usepackage{amsthm}
\usepackage{mathtools}

\usepackage{thmtools}
\usepackage{eucal} 
\usepackage{mathdots}

\usepackage{enumitem}
\usepackage{mathrsfs}  

\usepackage{tikz-cd}
\usepackage{adjustbox}
\usepackage{contour}

\usepackage{faktor}

\usepackage[colorlinks=true,linkcolor=red,citecolor=blue,urlcolor=green]{hyperref}
\usepackage[capitalise]{cleveref}

\numberwithin{equation}{section}
\crefname{equation}{}{}

\crefformat{section}{\S#2#1#3}
\crefmultiformat{section}{\S\S#2#1#3}{ and~#2#1#3}{, #2#1#3}{, and~#2#1#3}

\contourlength{0.8pt}

\newtheorem{theorem}{Theorem}[section]
\newtheorem{prop}[theorem]{Proposition}
\newtheorem{cor}[theorem]{Corollary}
\newtheorem{lemma}[theorem]{Lemma}

\theoremstyle{definition}
\newtheorem{dfn}[theorem]{Definition}
\newtheorem{rmk}[theorem]{Remark}

\newtheorem{ex}[theorem]{Example}

\DeclareMathOperator{\Hom}{\mathsf{Hom}}

\DeclareMathOperator{\Ext}{\mathsf{Ext}}
\DeclareMathOperator{\Tor}{\mathsf{Tor}}

\DeclareMathOperator{\RHom}{\mathbf{R}\mathsf{Hom}}

\newcommand{\Acal}{\mathcal{A}}
\newcommand{\Bcal}{\mathcal{B}}
\newcommand{\Ccal}{\mathcal{C}}

\newcommand{\Ecal}{\mathcal{E}}
\newcommand{\Fcal}{\mathcal{F}}
\newcommand{\Gcal}{\mathcal{G}}

\newcommand{\Ical}{\mathcal{I}}

\newcommand{\Pcal}{\mathcal{P}}
\newcommand{\Qcal}{\mathcal{Q}}
\newcommand{\Rcal}{\mathcal{R}}
\newcommand{\Scal}{\mathcal{S}}

\newcommand{\Vcal}{\mathcal{V}}
\newcommand{\Wcal}{\mathcal{W}}
\newcommand{\Xcal}{\mathcal{X}}
\newcommand{\Ycal}{\mathcal{Y}}

\newcommand{\Qbb}{\mathbb{Q}}

\newcommand{\Zbb}{\mathbb{Z}}

\newcommand{\D}{\mathsf{D}}

\newcommand{\Mod}[1]{\mathsf{Mod}\mbox{-}#1}
\newcommand{\lMod}[1]{#1\mbox{-}\mathsf{Mod}}
\newcommand{\lrMod}[2]{#1\mbox{-}\mathsf{Mod}\mbox{-}#2}

\renewcommand{\mod}[1]{\mathsf{mod}\mbox{-}#1}

\newcommand{\depth}{\mathsf{depth}}
\newcommand{\grade}{\mathsf{grade}}
\newcommand{\Kdim}{\mathsf{dim}}

\newcommand{\Spec}[1]{\mathsf{Spec}(#1)}
\newcommand{\mSpec}[1]{\mathsf{mSpec}(#1)}

\newcommand{\Ass}[1]{\mathsf{Ass}(#1)}
\newcommand*{\Perp}[1]{{}^{\perp_{#1}}}

\newcommand{\Prod}[1]{\mathsf{Prod}(#1)}
\newcommand{\Sub}[1]{\mathsf{Sub}(#1)}

\newcommand{\Ker}{\mathsf{Ker}}
\newcommand{\Coker}{\mathsf{Coker}}

\newcommand{\pp}{\mathfrak{p}}

\newcommand{\qq}{\mathfrak{q}}
\newcommand{\mm}{\mathfrak{m}}

\newcommand{\height}{\mathsf{height}}

\newcommand{{\tst}}{\textit{t}-}
\newcommand{\pd}{\mathsf{pd}}
\newcommand{\fd}{\mathsf{fd}}
\newcommand{\id}{\mathsf{id}}

\newcommand{\Rfd}{\mathsf{Rfd}}

\newcommand{\rfd}{\mathsf{rfd}}
\newcommand{\Findim}{\mathsf{Findim}}
\newcommand{\findim}{\mathsf{findim}}
\newcommand{\WFindim}{\mathsf{W.Findim}}

\newcommand{\newterm}[1]{\textit{#1}}

\newcommand{\coleq}{\vcentcolon=}

\title[Cotorsion pairs and Tor-pairs]
{Cotorsion pairs and Tor-pairs over commutative noetherian rings}

\author{Dolors Herbera}
\address[D. Herbera]{Departament de Matem\`atiques,
Universitat Aut\`onoma de Barcelona,  08193 Bellaterra
(Barcelona), Spain\newline
Centre de Recerca Matemàtica,  08193 Bellaterra
(Barcelona), Spain}
\email{dolors.herbera@uab.cat}

\thanks{The first author acknowledges the support of the Spanish State Research Agency, through the Severo Ochoa and María de Maeztu Program for Centers and Units of Excellence in R \& D (CEX2020-001084-M). She was also partially supported by the projects MIMECO  PID2020-113047GB-I00 and PID2023-147110NB-I00 financed by the Spanish Government  and  \emph{Laboratori d'Interaccions entre Geometria, \`Algebra i Topologia} (LIGAT) with reference number 2021 SGR 01015 financed by the Generalitat de Catalunya. The second author was supported by the GAČR project 23-05148S and the Academy of Sciences of the Czech Republic (RVO 67985840). The third author was also supported by GAČR project 23-05148S and \emph{Laboratori d'Interaccions entre Geometria, \`Algebra i Topologia} (LIGAT) with reference number 2021 SGR 01015 financed by the Generalitat de Catalunya.}

\author{Michal Hrbek}
\address[M. Hrbek]{Institute of Mathematics of the Czech Academy of Sciences, \v{Z}itn\'{a} 25, 115 67 Prague, Czech Republic}
\email{hrbek@math.cas.cz}

\author{Giovanna Le Gros}
\address[G. Le Gros]{{\it current address:} Institute of Mathematics of the Czech Academy of Sciences, \v{Z}itn\'{a} 25, 115 67 Prague, Czech Republic
\newline Charles University, Faculty of Mathematics and Physics, Department
of Algebra, Sokolovská 83, 186 75 Praha, Czech Republic. \newline 
Departament de Matem\`atiques,
Universitat Aut\`onoma de Barcelona,  08193 Bellaterra
(Barcelona), Spain}
\email{legros@math.cas.cz}


\begin{document}
\begin{abstract}
For a commutative noetherian ring, we classify all the hereditary cotorsion pairs cogenerated by pure-injective modules of finite injective dimension. The classification is done in terms of integer-valued functions on the spectrum of the ring. Each such function gives rise to  local depth conditions which describe the left-hand class in the corresponding cotorsion pair. 
Furthermore, we show that these cotorsion pairs correspond via an explicit duality to hereditary $\Tor$-pairs generated by modules of finite flat dimension. 
\end{abstract}
\maketitle
\section{Introduction}

There have been many successful classification theorems of pure-injective modules over specific rings. However, pure-injective modules over even a well-understood ring, such as a noetherian ring of global dimension greater than one, can have pathological behaviour \cite{JL89}. Given that describing them explicitly can be a daunting task, another approach is instead to associate a class of modules to a given collection of pure-injective modules. In this paper, we consider the left-orthogonal with respect to the $\Ext$ functor, in particular over commutative noetherian rings.

These classes have several appealing properties. They are known to be closed under direct limits, as well as having good approximation properties; specifically, they are covering classes. In other terminology, these classes are the left constituents of perfect complete cotorsion pairs, a non-trivial consequence of being cogenerated by pure-injective modules, a result which was used to prove affirmatively the long-standing question of whether flat covers necessarily exist \cite{BEE01}.    

Important examples of classes that are $\Ext$-orthogonal to a collection of pure-injective modules are the classes appearing in a $\Tor$-pair. $\Tor$-pairs were introduced as an analogue of cotorsion pairs, that is, pairs of classes which are orthogonal with respect to the $\Tor$ functor instead of the $\Ext$ functor. Via a Tor-Ext duality induced by the $\mathrm{Hom}$-$\otimes$ adjunction, the two classes in a $\Tor$-pair arise as a left-hand class of a complete cotorsion pair cogenerated by the character dual of a module, which is always pure-injective. 

Going back to general cotorsion pairs cogenerated by pure-injective modules, we recall that if there is a bound on the injective dimensions of the pure-injective modules and, additionally, one imposes the condition that the left-hand class of the cotorsion pairs is definable, that is, in particular closed under arbitrary products, then these cotorsion pairs are exactly the cotilting cotorsion pairs, see \cref{def-tor-pairs}. 
Such cotorsion pairs over commutative noetherian rings \cite{AHPST14}, and even over certain general commutative rings, \cite{H16, HS20, Baz07, BH21} have been classified in terms of various ring-theoretic data. In particular, the cotilting cotorsion pairs of cofinite type, which are the ones that appear as a $\Tor$-orthogonal class to a family of strongly finitely presented modules of finite projective dimension, are classified using the spectrum of the ring or, more specifically, in terms of Thomason subsets. Over commutative noetherian rings, all cotilting classes are of cofinite type \cite{AHPST14}. It is also interesting to note that in the characterisations in \cite{AHPST14,H16,HS20}, the Koszul complex with respect to a finitely generated ideal plays a pivotal role.

In this paper, we are principally interested in commutative noetherian rings. In Theorem~\ref{T:Torpair-characterisation}, we show that the cotorsion pairs cogenerated by pure-injective modules of finite injective dimension are classified by functions from the spectrum to the non-negative integers, which are bounded by the local depth of the prime. The left-hand constituent of the associated cotorsion pair consists precisely of those modules whose local depths are bounded by the corresponding function. As we detail in Corollary~\ref{T:definable}, our results nicely extend the classification of cotilting cotorsion pairs in the commutative noetherian case, which corresponds to those functions which are in addition order-preserving. We remark that in the follow-up work \cite{HHLG26}, we extend this result to the derived category of quasi-coherent sheaves over a noetherian scheme, where the notion of a $\Tor$-pair is generalised to that of a (pair of) semi-Bousfield classes.

Next, we pay attention to $\Tor$-pairs.  It follows that, over any given ring, there is a set of $\Tor$-pairs and the size of this set has a bound that depends on  the cardinality of the ring, we recall this in \cref{bound-tor-pairs}. By duality, the $\Tor$-orthogonal class to a class of modules of finite flat dimension is also  $\Ext$-orthogonal to a class of pure-injective modules of finite injective dimension. As an outcome of our work, we show that in the commutative noetherian setting, the converse is true as well, so that the $\Ext$-orthogonal classes we consider come from a $\Tor$-pair. Thus, our classification in Theorem~\ref{T:Torpair-characterisation} is also a classification of $\Tor$-pairs generated by modules of finite flat dimension. The choice of generators can be made explicit --- the function gives rise to a collection of suitable cocycle modules in localised Koszul complexes associated to each prime ideal, see \cref{function-to-torpair}.

The characterisation of $\Tor$-pairs over a commutative noetherian ring $R$ was posed as an open problem in the monograph of G\"{o}bel and Trlifaj \cite[\S 16.5, 3.]{GT12}. This was motivated by the aforementioned classification of cotilting classes, as well as the fact that if $R$ is hereditary then constituents of $\Tor$-pairs coincide with cotilting classes \cite[Theorem 16.31]{GT12}. In \cref{hered-tor-pair-reg-ring}, we show that if $R$ is regular then our main classification result covers all hereditary $\Tor$-pairs and all hereditary cotorsion pairs cogenerated by pure-injective modules. On the other hand, in the singular case there are hereditary $\Tor$-pairs which do not fit into our classification, see \cref{ex1}, \cref{ex2}, and \cref{opposite-tor}. Leaving the noetherian setting for a short while, we will also supplement this in \cref{vd-example} by producing a commutative semihereditary ring (in fact, a valuation domain) over which there are $\Tor$-pairs which contain a class which is not a cotilting class. 

Additionally, while cotilting classes originate in representation theory, the present work also extends ideas that were studied in commutative algebra and module theory. For example, the large and small restricted flat dimension introduced by Christensen, Frankild and Foxby \cite{CFF02} can be interpreted in terms of suitable $\Tor$-pairs, and their results on when these two dimension theories coincide can be interpreted as a special case of our classification, see \cref{s:rfd}.  

We begin this manuscript with some background in the aptly named \cref{s:preliminaries}, which begins over an arbitrary ring before specialising to commutative rings and then to commutative noetherian rings, introducing the machinery we will need at each stage. Next we introduce some preliminary results in \cref{s:prelimaryresults} on $\Tor$-pairs, cotorsion pairs and pure-injective modules over commutative noetherian rings. The bulk of the main results of this paper are in \cref{s:hered-tor-pair}, which all take place over a commutative noetherian ring, except for some examples. 
The main classification is stated in \cref{T:Torpair-characterisation}. Along the way we provide various examples of $\Tor$-pairs which do not fit various hypotheses. We specialise to regular rings in \cref{hered-tor-pair-reg-ring} and show that in this setting, our main result gives a classification of all hereditary $\Tor$-pairs. In the final \cref{s:rfd}, we show that a special case of our classification can be used to characterise almost Cohen--Macaulay rings in terms of small and large restricted flat dimensions being equal, providing a global version of \cite[Theorem 3.2]{CFF02}. This last argument also employs the recent study of the Govorov-Lazard Theorem in higher dimensions of \cite{HLG24}.

\subsection*{Acknowledgements}
The authors wish to thank the anonymous referee for a thorough reading of the manuscript and their suggestions, which greatly improved the readability. 

\section{Background}\label{s:preliminaries}
Let $R$ be an associative unital ring. We denote by $\Mod R$ the category of all right $R$-modules and by $\mod R$ the (full, isomorphism-closed) subcategory consisting of right $R$-modules which admit a resolution by finitely generated projective $R$-modules. We refer to the modules in $\mod R$  as  \newterm{strongly finitely presented modules}. Analogously, we denote by $\lMod R$ the category of all left $R$-modules.

Fix a class $\Ccal$. We let $\Sub \Ccal$ denote the class of modules that are isomorphic to a submodule of a module of $\Ccal$.  

\subsection{}\label{ss:dimensions} Let us denote by $\pd_R M$, $\id_R M$, and $\fd_R M$, respectively, the projective dimension, the injective dimension, and the flat dimension of an $R$-module $M$. For $n \geq 0$, let $\Pcal_n(R) = \{M \in \Mod R \mid \pd_R M \leq n\}$, $\Ical_n(R) = \{M \in \Mod R \mid \id_R M \leq n\}$, and $\Fcal_n(R) = \{M \in \Mod R \mid \fd_R M \leq n\}$ denote the subcategories of $\Mod R$ consisting of all modules of projective, injective, and  flat dimension, respectively, bounded above by $n$. We put $\Pcal_n^{<\aleph_0}(R) = \Pcal_n(R) \cap \mod R$. We use the notation $\Pcal(R) = \bigcup_{n \geq 0}\Pcal_n(R)$ for modules of finite projective dimension, similarly we put $\Fcal(R) = \bigcup_{n \geq 0}\Fcal_n(R)$, $\Ical(R) = \bigcup_{n \geq 0}\Ical_n(R)$ and $\Pcal^{<\aleph_0}(R) = \bigcup_{n \geq 0}\Pcal_n^{<\aleph_0}(R)$. 
We often omit the reference to the ring when the ring is clear from context, and we  write simply $\Pcal_n$, $\Ical_n$, $\Fcal_n$, $\Pcal$, $\Fcal$, or $\Pcal_n^{<\aleph_0}$ and $\Pcal^{<\aleph_0}$. 

For an $R$-module $M$, $\Omega_i(M)$ is defined to be an $i$th syzygy if $i>0$, or the $i$th minimal cosyzygy if $i<0$. By convention, $\Omega_0 (M) =M$, and $\Omega_i(M)$ is some fixed chosen representative up to a projective direct summand for $i>0$. We shall also use the following less usual notion. Let 
\[
\begin{tikzcd}
\cdots \arrow[r] &F_n \arrow[r,"d_n"] &  F_{n-1}  \arrow[r] & \cdots  \arrow[r] &F_0 \arrow[r,"d_0"] &M \arrow[r] &0,
\end{tikzcd}
\]
be a minimal flat presentation of $M$.
Then the \newterm{$i$th yoke}, denoted by $\Upsilon_i(M)$, of $M$ is defined to be $\Ker (d_{i-1})$ for $i \geq 1$, and the $0$th yoke is $M$ itself. 

For a class $\Xcal$ in $\Mod R$, we let 
$$\Xcal^{\top_{1}}= \{M \in \lMod R \mid \Tor^R_1(X,M) = 0, ~\forall X \in \Xcal\}$$ and 
$$\Xcal^{\top_{\geq i}}= \{M \in \lMod R \mid \Tor^R_j(X,M) = 0 ~\forall X \in \Xcal, ~\forall j\geq i\},$$ and for brevity, $$\Xcal^{\top_{}}\coleq \Xcal^{\top_{\geq 1}}.$$ 
Analogously, for a class $\Ycal$ in $\lMod R$, we can define the classes ${}^{\top_1}\Ycal$, ${}^{\top_{\geq i}}\Ycal$  and ${}^\top\Ycal$ in $\Mod R$. 

For a class $\Xcal$ in $\Mod R$, we let 
$$\Xcal^{\perp_{1}}= \{M \in \Mod R \mid \Ext_R^1(X,M) = 0, ~\forall X \in \Xcal\},$$ 
$$\Xcal^{\perp_{\geq i}}= \{M \in \Mod R \mid \Ext_R^j(X,M) = 0 ~\forall X \in \Xcal, ~\forall j\geq i\},$$ and for brevity, $$\Xcal^{\perp_{}}\coleq \Xcal^{\perp_{\geq 1}}.$$ 
Analogously, for a class $\Ycal$ in $\Mod R$ we can define the classes ${}^{\perp_1}\Ycal$, ${}^{\perp_{\geq i}}\Ycal$  and ${}^\perp\Ycal$ in $\Mod R$. 

\subsection{Homological formulae}\label{ss:hom-formulae}
The $\mathrm{Hom}$-$\otimes$ adjunction yields the following natural isomorphism 
$$\Ext _R^i(A, \mathrm{Hom}_S(B,C))\cong  \mathrm{Hom}_S(\Tor _i^R(A,B),C)$$
for any $i \ge 0$, where $A\in \Mod R$, $B \in \lrMod R S$, and $C \in \Mod S$ is injective (see \cite[Lemma 2.16]{GT12}). If $C_S$ is, in addition, an injective cogenerator then $\Ext _R^n(A, \mathrm{Hom}_S(B,C))=0$ if and only if $\Tor _n^R(A,B)=0$.

Recall that an $R$-module $M$ is \newterm{cotorsion} if $\Ext^1_R(F,M)$ vanishes for every $F \in \Fcal_0(R)$. 

Let $R$ and $S$ be rings. Consider the modules $A \in \Mod R$, $B \in \lrMod R S$ which is flat both as an $R$-module and as an $S$-module, and $C \in \Mod S$ a cotorsion $S$-module. Then, by the derived $\mathrm{Hom}$-$\otimes$ adjunction, there are the following natural isomorphisms of abelian groups for every $i\geq 0$.

\begin{equation}\label{eq:derived-tensor-hom}
\Ext^i_S(A\otimes_R B, C) \cong \Ext^i_R(A, \Hom_S( B, C))
\end{equation}
Indeed, there are natural isomorphisms
\begin{align*}
    \Ext^i_S(A\otimes_R B, C) &\cong 
    \Hom_{\D(S)}(A \otimes_R B,C[i])\\ & \cong \Hom_{\D(S)}(A \otimes_R^\mathbf{L} B,C[i]) \\ & \cong
    \Hom_{\D(R)}(A,\RHom_S(B,C)[i]) \\ & \cong \Hom_{\D(R)}(A,\Hom_S(B,C)[i]) \\ & \cong \Ext^i_R(A,\Hom_S(B,C))
\end{align*}
where the second isomorphism follows from the flatness of $B$ as an $R$-module, the third isomorphism is the derived $\Hom$-$\otimes$ adjunction, and the fourth isomorphism follows from the flatness of $B$ as an $S$-module and because $C$ is a cotorsion $S$-module.

For an alternative approach to the proof of \ref{eq:derived-tensor-hom}, choose a short exact sequence
\[0\to \Omega _1 (A)\to P\to A\to 0 \qquad (*)\]
with $P$ projective right $R$-module. As ${}_RB$ is flat, after applying the functor $-\otimes _RB$ to $(*)$ we also get a short exact sequence of right $S$-modules,  and its central term $P\otimes _RB$ is flat as a right $S$-module because $P_R$ is projective and $B_S$ is flat. Now apply the functor $\Hom _S( -, C)$, and the fact that  $C$ is cotorsion, yields the exact sequence
\begin{multline*}0\to \Hom _S(A\otimes _RB, C)\to \Hom _S(P\otimes _RB, C)\to \\ \Hom _S(\Omega_1 (A)\otimes _RB, C)\to \Ext^1_S(A\otimes_R B, C)\to 0. \qquad (**)\end{multline*}
Applying the functor $\Hom _R (-,\Hom_S(B,C)) $ to $(\ast)$ yields the exact sequence
\begin{multline*}0\to \Hom _R(A, \Hom_S(B,C))\to \Hom _R(P, \Hom_S(B,C))\to \\ \Hom _R(\Omega_1 (A) ,\Hom_S(B,C))\to \Ext^1_R(A, \Hom_S(B,C))\to 0. \qquad (**')\end{multline*}
The  $\mathrm{Hom}$-$\otimes$ adjunction yields a natural isomorphism between the three leftmost terms of $(**)$ and the three leftmost terms of $(**')$. Therefore, there is also a natural isomorphism 
$\Ext^1_S(A\otimes_R B, C) \cong \Ext^1_R(A, \Hom_S( B, C))$. The isomorphism for any $i>1$ can be deduced from a standard dimension-shifting argument.

\subsection{Pure-injective modules} \label{ss:pure-injective}
We refer the reader to \cite[\S 2.1]{GT12} or the monograph \cite{Prest} for more details related to this subsection. A \newterm{pure embedding}, or \newterm{pure monomorphism} of right $R$-modules is a monomorphism $\mu \colon A \to B$ which remains exact under the tensor product $(- \otimes_R C)$ for any $C \in \lMod R$, or equivalently, the short exact sequence 
\begin{equation}\label{e:equation_pure}
0 \to A \overset{\mu} \to B \to B/\mu(A) \to 0
\end{equation}
remains exact under $\Hom_R(S,-)$ for all finitely presented right $R$-modules $S$. If $\mu$ is a pure embedding, then the short exact sequence \cref{e:equation_pure} is called a \newterm{pure-exact sequence}, the epimorphism $B \to B/\mu(A)$ in \cref{e:equation_pure} is called a \newterm{pure epimorphism}, and $\mu(A)$ is a \newterm{pure submodule} of $B$.

A module $N$ is \newterm{pure-injective} if the induced map $\Hom_R(B,N) \to \Hom_R(A, N)$ is an epimorphism for any pure embedding $A \to B$. In other words, the pure-injective modules are the modules which are injective with respect to pure monomorphisms. We let $\Pcal\Ical_0$ denote the class of pure-injective modules. A fact we will use throughout is that cosyzygies of pure-injective modules are pure-injective, \cite[Lemma 6.20]{GT12}.

We give some examples of pure-injective modules. For a module $M \in \Mod R$, we denote the \newterm{character module} of $M$ by $M^+ \coleq \Hom_\Zbb(M, \Qbb/\Zbb)$. If $(R,\mm,k)$ is a local commutative noetherian ring then we also consider the \newterm{Matlis dual} $M^\vee \coleq \Hom_R(M,E(k))$ of an $R$-module $M$, where $E(k)$ is the injective hull of $k$. Note that both $M^+$ and $M^\vee$ are pure-injective $R$-modules. Moreover, the evaluation map yields pure monomorphisms $M\to M^{++}$  and, for suitable $R$, $M\to M^{\vee \vee}$, so any pure-injective module is a direct summand of a dual module.

All pure-injective modules are cotorsion, and in fact a module is flat if and only if $\Ext_R^1(F, N)$ vanishes for every pure-injective module $N$. Indeed, by \cref{ss:hom-formulae},
 for any left $R$-module $M$,  $\Ext_R^1(F, M^+) = 0$ if and only if $\Tor_1^R(F,M)=0$.

The next homological formula will prove useful.
For an $R$-module $N$ and any directed system $\{M_i\}_{i \in I}$, there is a natural isomorphism of abelian groups. 
\[
\Hom_R(\varinjlim_I M_i, N) \cong \varprojlim_I \Hom_R(M_i, N) 
\]
By a well known result of Auslander, if $N$ is additionally pure-injective, then the same result also holds for higher $\Ext$-groups, see \cite[Lemma 6.28]{GT12}.
\[
\Ext^j_R(\varinjlim_I M_i, N) \cong \varprojlim_I \Ext^j_R(M_i, N) 
\]

\begin{lemma}\label{ext-flat-pi}
    Let $R$ be a ring, $F \in \Mod R$ a flat $R$-module, $Y\in \lrMod R S$ and $N \in \Mod S$ a pure-injective module. Then for every $j\geq0$, there is a natural isomorphism of abelian groups.
    \[
        \Ext^j_S(F\otimes_RY, N) \cong \Hom_R(F, \Ext^j_R(Y, N))
    \]
\end{lemma}
\begin{proof}
    By the classical result of Govorov and Lazard, a flat module $F$ can be written as $F = \varinjlim_{i \in I} F_i $ where each $F_i$ is a finitely generated free $R$-module. Then there are the following natural isomorphisms.
    \begin{align*}
\Ext^j_S(F \otimes_RY, N) 	&\cong  \Ext^j_S((\varinjlim_{i \in I} F_i) \otimes_R Y, N) \\
						&\cong  \Ext^j_S(\varinjlim_{i \in I} (F_i\otimes_RY), N)  \\
						&\cong  \varprojlim_{i \in I}\Ext^j_S( F_i\otimes_RY, N) \\
                            &\cong  \varprojlim_{i \in I}\Hom_R(F_i,\Ext^j_S( Y , N)) \\
                            &\cong  \Hom_R(F,\Ext^j_S( Y , N)) 
\end{align*}
\end{proof}

 \subsection{Cotorsion pairs.}\label{ss:cotorsion-pairs}

A \newterm{cotorsion pair} is a pair of classes $(\Acal, \Bcal)$ in $\Mod R$, such that $\Acal = {}^{\perp_1} \Bcal$ and $\Bcal = \Acal^{\perp_1}$. A cotorsion pair is \newterm{hereditary} if moreover $\Acal = {}^{\perp} \Bcal$ and $\Bcal = \Acal^{\perp}$, that is, all the higher $\Ext$-groups vanish as well. 

Let $\Xcal$ be a class in $\Mod R$. The cotorsion pair \newterm{generated} by $\Xcal$ is the cotorsion pair $({}^{\perp_1}(\Xcal^{\perp_1}),\Xcal^{\perp_1})$. The cotorsion pair \newterm{cogenerated} by $\Xcal$ is the cotorsion pair $({}^{\perp_1}\Xcal,({}^{\perp_1}\Xcal)^{\perp_1})$. 

The left-hand class in a cotorsion pair is closed under direct sums and extensions, and contains all the projective modules. The right-hand class in a cotorsion pair is closed under products and extensions, and contains all the injective modules. A class $\Xcal$ in $\Mod R$ is \newterm{resolving} if $\Xcal$ is closed under extensions, contains the projectives, and is closed under kernels of epimorphisms. Dually, a class $\Xcal$ in $\Mod R$ is \newterm{coresolving} if $\Xcal$ is closed under extensions, contains the injectives, and is closed under cokernels of monomorphisms. 
A cotorsion pair $(\Acal, \Bcal)$ is hereditary if and only if $\Acal$ is a resolving class if and only if $\Bcal$ is a coresolving class. 

A cotorsion pair $(\Acal, \Bcal)$ is \newterm{complete} if $\Acal$ is special precovering, or equivalently, if $\Bcal$ is special preenveloping, see \cite[\S 5 and \S 6]{GT12} for details.

\subsection{$\Tor$-pairs.}\label{ss:Tor-pairs}
A \newterm{$\Tor$-pair} is a pair of classes $(\Ecal, \Ccal)^\top$ with $\Ecal$ in $\Mod R$ and $\Ccal$ in $\lMod R$, such that $\Ecal = {}^{\top_1} \Ccal$ and $\Ccal = \Ecal^{\top_1}$. A $\Tor$-pair is \newterm{hereditary} if moreover $\Ecal = {}^{\top} \Ccal$ and $\Ccal = \Ecal^{\top}$, that is, all the higher $\Tor$-groups vanish as well. In particular, $(\Fcal_n, \Fcal_n^\top)^\top$ forms a hereditary $\Tor$-pair for every $n \geq 0$. Indeed, the $n$-th syzygy module $\Omega^n(M)$ of any left $R$-module $M$ belongs to $\Fcal_n^\top$ by a dimension-shift argument. On the other hand, if $F \in {}^\top\Omega^n(M)$ for any left $R$-module $M$ then it is straightforward to verify that $F \in \Fcal_n$, establishing that ${}^\top((\Fcal_n)^\top) = \Fcal_n$. Then $(\Fcal_n, \Fcal_n^\top)^\top$ forms a $\Tor$-pair, and since $\Fcal_n$ is a resolving subcategory of $R$-modules, this $\Tor$-pair is clearly hereditary.

Let $\Xcal$ be a class in $\Mod R$. The $\Tor$-pair \newterm{generated} by $\Xcal$ is the $\Tor$-pair of the form $({}^{\top_1}(\Xcal^{\top_1}),\Xcal^{\top_1})^\top$. If $\Xcal$ is closed under yokes, then the $\Tor$-pair it generates is hereditary. Both classes in a $\Tor$-pair are closed under direct limits, pure submodules, pure epimorphisms and arbitrary direct sums. Over commutative rings, it is clear that if $(\Ecal, \Ccal)^\top$ is a $\Tor$-pair, then also $(\Ccal, \Ecal)^\top$ is a (distinct) $\Tor$-pair.

\subsection{The lattice of $\Tor$-pairs and cotorsion pairs}\label{ss:lattice-of-pairs}

The collection of cotorsion pairs in $\Mod R$ is  partially ordered by inclusion in their first component, and moreover forms a complete lattice, denoted $L_{\Ext_1}$. The cotorsion pairs $(\Pcal_0, \Mod R)$ and $(\Mod R, \Ical_0)$ form the smallest and largest element respectively. In general, the lattice $L_{\Ext_1}$ is class-sized \cite[Example 5.16]{GT12}.

Similarly, the collection of $\Tor$-pairs partially ordered by inclusion in their first component form a complete lattice, denoted $L_{\Tor_1}$. The $\Tor$-pairs $(\Fcal_0, \lMod R)^\top$ and $(\Mod R, \Fcal_0)^\top$ form the smallest and largest element respectively. The lattice $L_{\Tor_1}$ can be canonically embedded into $L_{\Ext_1}$ via the following assignment.

\begin{lemma}\cite[Lemma 2.16(b), Lemma 5.17]{GT12}\label{tor-induce-cotor}
Let $(\Ecal, \Ccal)^\top$ be a $\Tor$-pair. 
\begin{enumerate}
    \item There is a cotorsion pair $(\Ecal, \Ecal^{\perp_1})$ of right $R$-modules which is cogenerated by $\Ccal^+ \coleq \{C^+ \mid C \in \Ccal\}$.
    \item There is a cotorsion pair $(\Ccal, \Ccal^{\perp_1})$ of left $R$-modules which is cogenerated by $\Ecal^+ \coleq \{E^+ \mid E \in \Ecal\}$. 
\end{enumerate}
Moreover, if $(\Ecal, \Ccal)^\top$ is hereditary, then so are $(\Ecal, \Ecal^{\perp_1})$ and $(\Ccal, \Ccal^{\perp_1})$.
\end{lemma}

In contrast to the case of cotorsion pairs, the lattice of $\Tor$-pairs forms a set.
\begin{cor}\cite[Corollary 6.24]{GT12} \label{bound-tor-pairs}
Let $R$ be a ring and $\kappa = |R| + \aleph_0$. Then $|L_{\Tor_1}| \leq 2^{2^\kappa}$.
\end{cor}

The question remains of whether one can classify all the $\Tor$-pairs over a particular ring; in the case of commutative noetherian rings, this was formulated in \cite[\S 16.5, 3.]{GT12}. Some specific examples are known: $\Tor$-pairs over a Dedekind domain have been completely characterised in terms of subsets of the spectrum, and, the $\Tor$-pairs over a right hereditary ring have been classified in terms of resolving subcategories of $\mod R$, see \cite[Theorem 16.29 and 16.31]{GT12}. We will show that it is not possible to classify all the $\Tor$-pairs in such a way over an arbitrary commutative noetherian ring, that is, either in terms of the spectrum or in terms of resolving subcategories. Instead, we restrict to a specific collection of $\Tor$-pairs. 

\subsection{Hereditary cotorsion pairs cogenerated by pure-injective modules} Let $(\Ccal,\Wcal)$ be a hereditary cotorsion pair cogenerated by a class $\Qcal$ of pure-injective $R$-modules, that is, $\Ccal = \Perp{}\Qcal$. Then $(\Ccal,\Wcal)$ is complete and perfect, the latter meaning that $\Ccal$ is closed under direct limits, see \cite[\S 6]{GT12}. By \cref{tor-induce-cotor}$(2)$, a hereditary $\Tor$-pair $(\Ecal,\Ccal)^\top$ gives rise to a hereditary cotorsion pair $(\Ccal,\Wcal)$ cogenerated by pure-injectives via the adjunction formula $\Ccal = \Perp{>0}(\Ecal^+)$. This in fact gives a lattice isomorphism between the lattice $L_{\Tor_1}$ of  hereditary $\Tor$-pairs and the sublattice of $L_{\Ext^1}$ consisting of those hereditary cotorsion pairs which are cogenerated by character dual modules, which in turn embeds into the sublattice of $L_{\Ext^1}$ consisting of those hereditary cotorsion pairs which are cogenerated by pure-injective modules.

\begin{rmk}\label{rmk:pi-cot-pair-tor-class}
    We do not know whether each hereditary cotorsion pair cogenerated by pure-injective modules is induced by a hereditary $\Tor$-pair in the above sense, or equivalently, is cogenerated by character duals. In \cref{s:hered-tor-pair}, we show that this holds for a commutative noetherian ring when restricted to pure-injective modules of finite injective dimension. In \cref{hered-tor-pair-reg-ring}, we show that this covers all cotorsion pairs cogenerated by pure-injective modules provided the ring is, in addition, regular. 
    
    Outside of commutative noetherian rings, one could ask the same question for valuation domains. In this case, each pure-injective module is of injective dimension at most one. Bazzoni proved in \cite[Theorem 6.11]{Baz15} that the left-hand class in a cotorsion pair cogenerated by a cotilting module is a class in a $\Tor$-pair, see \cref{ss:cotilting} for results on cotilting modules. However, we do not know if this is necessarily the case for an arbitrary pure-injective modules, even of finite injective dimension, or for a cotilting module over more general families of rings.

    The question is similar in shape to asking whether each cohomological Bousfield class in the derived module category is a homological Bousfield class. For details, we refer to Stevenson's paper \cite{Ste14}, in which he proves  that there is a commutative von Neumann regular ring whose derived module category has a cohomological Bousfield class  which is not a homological Bousfield class, see \cite[Corollary 4.12]{Ste14}.
\end{rmk}

\subsection{Generation of $\Tor$-pairs}\label{ss:generating-Tor-pairs}
We call a class $\Ccal$ of $R$-modules \newterm{definable} if it is closed under products, direct limits, and pure submodules. It is straightforward to see that the $\Tor$-orthogonal of a set of finitely presented modules is a definable class. The converse does not hold, however, the following theorem describes exactly when this happens. An analogous result for cotorsion pairs is known to hold, see \cite[Theorem 6.1]{S18}. 

Before stating the result, we need to recall the concept of a relative Mittag-Leffler module. Let $M$ be a right module over a ring $R$, and let
$\Ccal$ be a class of left $R$-modules. We say that $M$ is a
\newterm{$\Ccal$-Mittag-Leffler module} if the canonical map
\[\rho \colon  M\otimes_R \prod _{i\in I}C_i\to \prod _{i\in
I}\left(M\otimes_RC_i\right)\] is injective for any family $\{C_i\}_{i\in
I}$ of modules in $\Ccal$.

\begin{theorem}\cite[Theorem 3.13]{H14}\label{tc-for-tor}
 Let $\Xcal$ be a class of right $R$-modules and let $\Ccal = \Xcal^{\top_1}$.Then $\Ccal$ is definable if and only if there exists a class of countably presented right $R$-modules $\Scal$ such that
\begin{enumerate}
    \item $\Xcal \subseteq \varinjlim \Scal $, 
    \item $\Scal^{\top_1}= \Ccal$, and
    \item the first syzygy of any module in $\Scal$ is $\Ccal$-Mittag-Leffler.
\end{enumerate}
In particular, if $(\Ecal, \Ccal)^\top$ is a (not-necessarily hereditary) $\Tor$-pair such that $\Ccal$ is definable, then there exists a set of countably presented modules $\Scal$ which generates the $\Tor$-pair and $ \Ecal$ coincides with the direct limit of its countably presented constituents.
\end{theorem}

\subsection{Cotilting classes.}\label{ss:cotilting}

\begin{dfn}\label{dfn-cotilting}
A left $R$-module $C$ is an \newterm{$n$-cotilting module} if the following conditions hold. 
\begin{itemize}
	\item[(C1)] $\id (C) \leq n$
	\item[(C2)] $\Ext^i_R(C^I, C)=0$ for all $i>0$ and any set $I$. 
	\item[(C3)] There exists a $\Prod C$-resolution of the injective cogenerator $W$ of $\Mod R$. That is, there is an exact sequence of the following form where $C^i \in \Prod C$ for each $0 \leq i \leq n$.
		\[
		\begin{tikzcd}[cramped, sep=small]
		0 \arrow[r] &C^n \arrow[r] & \cdots \arrow[r] & C^1 \arrow[r] &C^0 \arrow[r] &W \arrow[r] &0.
		\end{tikzcd}
		\]
\end{itemize}
The class ${}^\perp C$ is called the \newterm{(n-)cotilting class} associated to $C$. In particular, a cotilting module is always a pure-injective module of finite injective dimension \cite{Baz03, Sto06}.	
\end{dfn}
Equivalently, a class of modules $\Ccal$ is a cotilting class if and only if it is resolving, covering, closed under direct products and direct summands, and $\Ccal^\perp \subseteq \Ical_n(R)$ (see \cite[Theorem 15.9]{GT12}). In particular, cotilting classes are definable classes. See \cite[\S 15]{GT12} for details.

A  class $\Xcal$ of left $R$-modules is of \newterm{cofinite type} if there exists a set $\Scal \subseteq \mod R \cap \Pcal_n$ of right $R$-modules such that $\Scal^\top = \Xcal$. All classes of cofinite type are cotilting classes, but not all cotilting classes are of cofinite type, even over commutative rings, see for example \cite[Example 15.33]{GT12}. However, over commutative noetherian rings, all cotilting classes are of cofinite type \cite[Theorem 4.2]{AHPST14}.

By the following proposition, if $(\Ecal,\Ccal)^\top$ is a hereditary $\Tor$-pair generated by a class of modules of flat dimension at most $n\geq 0$ with $\Ccal$ a definable class, then $\Ccal$ is an $n$-cotilting class. 

\begin{prop}\cite[Proposition 3.14]{AHPST14},\cite[Theorem 2.3]{AHT04}\label{def-tor-pairs}
Let $R$ be any ring and $(\Ecal, \Ccal)^\top$ be a hereditary $\Tor$-pair with $\Ecal \subseteq \Fcal_n$ for some $n \geq 0$. Then $\Ccal$ is definable if and only if $\Ccal$ is an $n$-cotilting class.

Moreover, $\Ccal$ is an $n$-cotilting class of cofinite type if and only if $\Ecal = \varinjlim \Ecal^{<\aleph_0} = {}^\top\big((\Ecal^{<\aleph_0})^\top\big)$. 

\end{prop}
\begin{proof}

For the first statement, we comment that the $n$-th syzygy of a module in $\Ecal \subseteq \Fcal_n$ is flat, so will belong to $\Ccal$. 

The second equivalence holds as an $n$-cotilting class is of cofinite type if and only if there exists $\Scal \subseteq  \Pcal_n^{<\aleph_0}$ such that $\Scal^\top = \Ccal$, so in particular $\Ecal ={}^\top\big((\Ecal^{<\aleph_0}))^\top\big)$. That ${}^\top\big((\Ecal^{<\aleph_0})^\top\big) = \varinjlim (\Ecal^{<\aleph_0})$ holds by \cite[Theorem 2.3]{AHT04}.
\end{proof}

\begin{rmk}\label{rmk:valuation-domain-bazz}
     Bazzoni showed that over a valuation domain, there exist cotilting classes which are  not of cofinite type, \cite[Proposition 4.5]{Baz07}. Thus, there exist definable classes in hereditary $\Tor$-pairs generated by modules of bounded flat dimension which are not of cofinite type, even over commutative (semihereditary) rings.
     
     Additionally, as already mentioned in \cref{rmk:pi-cot-pair-tor-class}, by \cite[Theorem 6.11]{Baz15}, every cotilting class over a valuation domain is a class in a $\Tor$-pair. As valuation domains have weak global dimension at most one, every $\Tor$-pair is hereditary and generated by modules of flat dimension at most one. 

    However, the following question remains: Does it follow from \cref{tc-for-tor} that every cotilting class over any ring is of cocountable type? Here, a cotilting class $\Ccal$ is said to be of \newterm{cocountable type} if there is a set $\Scal$ of countably presented $R$-modules such that $\Ccal = \Scal^\top$. This holds for valuation domains and commutative noetherian rings. What is missing is to show that the left-hand class in a cotorsion pair cogenerated by a cotilting module is a class in a $\Tor$-pair, see \cref{rmk:pi-cot-pair-tor-class}. 
\end{rmk}

A \newterm{cotilting $\Tor$-pair} is a hereditary $\Tor$-pair $(\Ecal,\Ccal)^\top$ generated by a class of modules of finite flat dimension and where the right-hand class is a definable class. By Proposition~\ref{def-tor-pairs}, if a cotilting $\Tor$-pair is generated by a class of modules of flat dimension at most $n$, then it is an \newterm{$n$-cotilting $\Tor$-pair}. 

In  \cref{C:R_p-cotilting} we shall see  another consequence of \cref{def-tor-pairs} over commutative noetherian rings. Namely, given a cotilting $\Tor$-pair $(\Ecal,\Ccal)^\top$, \cref{def-tor-pairs} yields that $\Ccal_\pp$ is a cotilting class for all $\pp \in \Spec R$. 

Finally, let us also mention that in \cref{T:definable} we will show that, over a commutative noetherian ring, a cotilting $\Tor$-pair is a $\Tor$-pair generated by a class of finitely generated modules of finite projective dimension.

\begin{rmk}\label{rmk:cotilting-Tor-pairs}
    In view of \cref{def-tor-pairs}, the assignment $\Ccal \mapsto (\Ecal,\Ccal)^\top$ constitutes a bijection between $n$-cotilting classes and $n$-cotilting $\Tor$-pairs. However, our definition of a cotilting $\Tor$-pair is slightly more general than the usual definition of a cotilting class of \cref{dfn-cotilting}. Indeed, there is no uniform bound imposed on the flat dimensions of the generators of a cotilting $\Tor$-pair $(\Ecal,\Ccal)^\top$. In the case $R$ is a commutative noetherian ring of infinite Krull dimension, this will allow us later to formulate \cref{T:definable} in a more general form without having to impose a uniform bound on the order-preserving functions on $\Spec R$. We also remark that the class $\Ccal$ fitting into such a general cotilting $\Tor$-pair provides an example of an \newterm{$\infty$-cotilting class} as introduced in \cite{PS21}, see also \cite[Remark 5.7]{HNS} for a discussion about $\infty$-cotilting classes relevant to this setting
\end{rmk}

By \cite[Theorem 16.31]{GT12}, over a right hereditary ring all $\Tor$-pairs are classified in terms of resolving subcategories of $\mod R$. This implies that every class in a $\Tor$-pair over a commutative hereditary ring is a cotilting class of cofinite type. We show that this is not necessarily the case over semihereditary commutative rings, in particular over a commutative valuation domain. Thus, in this setting, there exist classes in $\Tor$-pairs which arise as left-$\Ext$-orthogonals of pure-injective modules of finite injective dimension that are not cotilting classes. 

\begin{ex}\label{vd-example}
Let $R$ be a (rank one) maximal valuation domain with idempotent maximal ideal $\mm$, and field of fractions $Q$. Let    $v \colon Q \to \Gamma$ be the associated valuation. As there is an idempotent prime (maximal) ideal, there is a subgroup of $v(Q)$ order-isomorphic to the rationals $\Qbb$. Recall that over a valuation domain, every module is of flat dimension at most one, and all pure-injective modules are of injective dimension at most one. In particular, $R$ being maximal implies that all ideals of $R$ (including the regular module $R$) are pure-injective $R$-modules, see \cite[Theorem XIII.5.2 and Theorem XV.3.2]{FS01}.

We now begin the construction of the desired $\Tor$-pair. By \cite[Proposition 4.5]{Baz07}, over a rank one, maximal, valuation domain, the module $C \coleq Q \oplus R \oplus R/\mm$ is a $1$-cotilting module not of cofinite type, and the cotilting class ${}^\perp C$ is exactly the class of modules which are extensions of flat modules by direct sums of copies of the residue field $R/\mm$, that is 
\[
\Ccal \coleq {}^\perp C = \{X \in \Mod R \mid 0 \to (R/\mm)^{(\alpha)} \to X \to F \to 0, F \in \Fcal_0( R)\}.
\]

By the construction of the cotilting class $\Ccal$, $\Ccal^\top = R/\mm^\top$, so we conclude that $(\Ccal, R/\mm^\top)^\top$ is a hereditary $\Tor$-pair by the facts stated in \cref{rmk:valuation-domain-bazz}. We claim that $R/\mm^\top$ is not closed under products, so it is not a cotilting class. To that aim, consider the collection of modules $\{R/s\mm\}_{s \in R \setminus \{0\}}$. We now see that the $R$-module $\Tor_1^R(R/s\mm, R/\mm)=0 $ for each nonzero $s$. First note that $\Tor_1^R(R/s\mm, R/\mm)\cong s\mm \otimes_R(R/\mm)$ induced by the connecting homomorphism of $(- \otimes_RR/\mm)$ applied to $0 \to s\mm \to R \to R/s\mm \to 0$. Then $s\mm \otimes_R(R/\mm)\cong (s\mm)/(s\mm^2) $ necessarily vanishes as $\mm$ is idempotent. 

Note also that for a non-zero element $s \in \mm$,  $ \Tor^R_1(R/sR, R/\mm) $ is isomorphic to $R/\mm$ by applying the appropriate functors to the short exact sequence $0 \to R \overset{\sigma}\to R \to R/sR \to 0$, where $\sigma$ is multiplication by the element $s$. 

Now fix a non-zero element $s \in \mm$. Then we claim that $sR = \bigcap_{v(t)<v(s)}t\mm= \bigcap_{v(t)<v(s)}tR$. The inclusion ``$\subseteq$'' is clear. For ``$\supseteq$'', we show the contrapositive. Take $r \notin sR$. Then $sR \subsetneq rR$, or $v(s)>v(r)$. As, by assumption, the value group $\Gamma = \Gamma(R)$ is  dense in $\mathbb{R}$ there exists an element $t$ of $R$ such that $v(s)>v(t)>v(r)$, so in particular $r \notin tR$ so $r \notin \bigcap_{v(t)<v(s)}tR$.

From the previous paragraph, we can conclude that there is a monomorphism $0 \to R/sR \to \prod_{v(t)<v(s)}R/t\mm$. Thus in particular since $\Tor_1^R(R/sR, R/\mm) \neq 0$, $\Tor_1^R(\prod_{v(t)<v(s)}R/t\mm, R/\mm) \neq 0$, so ${}^\perp R/\mm ={}^\top R/\mm$ is not closed under products.

In particular, we have just shown that $\mm$ is not Mittag-Leffler with respect to the family $\{R/t\mm\}_{v(t)<v(s)}$. Indeed, applying the functors $R/t\mm\otimes _R-$, for any $t\in R$ such that $v(t)<v(s)$, and $\left(\prod _{v(t)<v(s)}R/t\mm\right)\otimes _R-$ to the presentation
\[0\to \mm \to R \to R/\mm \to 0\] we get the following commutative diagram with exact rows, where we let $v\coloneqq v(s)$.

$$\begin{tikzcd}[column sep=2ex] 0\arrow[r]&\Tor _1^R(\displaystyle\prod _{v(t)<v}R/t\mm, R/\mm )  \arrow[r] \arrow[d, "f"]&\left(\displaystyle\prod _{v(t)<v}R/t\mm\right)\otimes _R
\mm \arrow[r] \arrow[d, "g"]&\left(\displaystyle\prod _{v(t)<v}R/t\mm\right)\otimes_R R\arrow[d, "h"]\\
0\arrow[r] & \displaystyle\prod _{v(t)<v}\left( \Tor _1^R(R/t\mm, R/\mm )\right)\arrow[r]&  \displaystyle\prod _{v(t)<v}\left( R/t\mm\otimes_R\mm\right) \arrow[r] &  
\displaystyle\prod _{v(t)<v}\left( R/t\mm\otimes_RR\right)\end{tikzcd}
$$
The vertical arrows are the natural transformations induced by commuting products with the tensor product, cf. the definition  before Theorem~\ref{tc-for-tor}. We have just shown that $f$ is not injective, and since $h$ is bijective,  $\mathrm{Ker}\,g\cong \mathrm{Ker}\,f$. Therefore, $\mm$ is not Mittag-Leffler with respect to the family $\{R/t\mm\}_{v(t)<v(s)}$, as claimed.
\end{ex}

We will return to the discussion of cotilting classes over commutative rings and over commutative noetherian rings in \cref{ss:cotiltingcomm}.

\subsection{Commutative rings: grade, depth and Koszul complexes.}\label{ss:Koszul}

We  define the grade with respect to an ideal for a general commutative ring $R$. We  follow the convention of using  \emph{grade} for general commutative rings and ideals, and reserving \emph{depth} for local rings and their maximal ideal. The following definition makes this precise.

\begin{dfn} \label{def:depth}
Let $R$ be a commutative ring. For a fixed ideal $J$ of $R$, we let the \newterm{grade} of $M$ relative to $J$, denoted $\grade_R (J;M)$, be the non-negative integer $\inf \{i \mid \Ext^i_R(R/J,M) \neq 0\}$. If $\Ext^i_R(R/J,M) = 0$ for all $i \geq 0$, then we set $\grade_R (J;M) = \infty$. 

If $R$ is local with maximal ideal $\mm$, then the \newterm{depth} of a module $M$ is the grade of $M$ with respect to the maximal ideal. That is,  $\depth_R (M) \coloneqq  \grade_R(\mm;M)$. \end{dfn}

The Koszul complex is grade-sensitive. We introduce the concepts and results we need to explain this following   \cite{BH98}, \cite[\S 1]{BourAC} and \cite{HS20}.

\medskip

Let $x \in R$. Then we let $K_\bullet(x)$ denote the complex
\[
\begin{tikzcd}
0 \arrow[r] &R \arrow[r, "(-)\cdot x"] & R  \arrow[r] &0, 
\end{tikzcd}
\]
with the leftmost $R$ in degree $1$ and the rightmost $R$ in degree $0$. Let $\mathbf{x}\coleq x_1, \dots, x_\ell$, $\ell >0$ a sequence of elements in $R$. Then we let $K_\bullet(\mathbf{x})$ denote the tensor product of the complexes $K_\bullet(x_i)$, that is
\[
K_\bullet(\mathbf{x}) \coleq \underset{1 \leq i \leq \ell}\bigotimes K_\bullet(x_i).
\]
All entries in this complex are finitely generated free modules and are nonzero only in homological degrees $0, 1, \dots, \ell$. 

Let $M$ be an $R$-module. The \newterm{Koszul (chain) complex} of $ \mathbf{x}$ with coefficients in $M$ is the chain complex $K_\bullet(\mathbf{x}) \otimes_RM$, denoted by $K_\bullet(\mathbf{x};M)$. The $i$th \newterm{Koszul homology} of $ \mathbf{x}$ with coefficients in $M$ is the $i$th homology of the chain complex $K_\bullet(\mathbf{x};M)$, and is denoted by $H_i(\mathbf{x};M)$. Clearly, $K_\bullet(\mathbf{x})\cong K_\bullet(\mathbf{x};R)$ and we denote by $H_i(\mathbf{x})$ the homology of $K_\bullet(\mathbf{x}) $.

The \newterm{Koszul cochain complex} of $ \mathbf{x}$ with coefficients in $M$ is the cochain complex $\Hom_R(K_\bullet(\mathbf{x}), M)$, denoted by $K^\bullet(\mathbf{x};M)$, and is nonzero only in cohomological degrees $0, 1, \dots, \ell$. The $i$th \newterm{Koszul cohomology} of $ \mathbf{x}$ with coefficients in $M$ is the $i$th cohomology of the cochain complex $K^\bullet(\mathbf{x};M)$, and is denoted by $H^i(\mathbf{x};M)$.
Similarly as in the dual, we denote by $K^\bullet(\mathbf{x}) $ the cochain complex $K^\bullet(\mathbf{x};R)$, and by $H^i(\mathbf{x})$ the cohomology of $K^\bullet(\mathbf{x}) $.

By \cite[Proposition 1.6.10]{BH98}, the complexes $K_\bullet(\mathbf{x};M)$ and $K^\bullet(\mathbf{x};M)$ are isomorphic, up to a relabelling of the cohomological degrees and a shift by $\ell$. In particular, $H_i(\mathbf{x};M) \cong H^{\ell-i}(\mathbf{x};M)$ for all $i$. 

Let $\mathbf{x}\coleq x_1, \dots, x_\ell$ and $\mathbf{y}\coleq y_1, \dots, y_m$ with $m,\ell >0$ be two sequences which generate an ideal $I$ of $R$. Then $H^i(\mathbf{x};M) = 0$ for all $i \leq n$ if and only if $H^i(\mathbf{y};M) = 0$ for all $i \leq n$, see \cite[Proposition 3.4]{HS20} which uses \cite[Corollary 1.6.2 and Corollary 1.6.10(d)]{BH98}. Since we are only interested in when the cohomology vanishes, we shall abuse the notation and denote $H^i(\mathbf{x};M)$ simply by $H^i(I;M) $.

The following Proposition will be essential for our purposes. It is implicit in the proof of \cite[p. 5, Théorème 1]{BourAC}, and moreover allows $J$ to be an infinitely generated ideal. The proposition is stated and proved using derived categories in \cite[Corollary 3.10]{HS20}. 

\begin{prop} \label{prop:depth}
If $R$ is a commutative ring, and let $J$ be an ideal of $R$. Then, for any $n>0$, 
\[
\bigcap_{i =0}^{n-1} \Ker \Ext^i_R(R/J, -) = \bigcap_{i =0}^{n-1}\Ker H^i(J;-)
\]
as subcategories of $\Mod R$. 
\end{prop}
\begin{proof}
We comment on how the statement follows from  the proof of \cite[p. 5, Théorème 1]{BourAC}. Explicitly, it is first shown that $$\bigcap_{i =0}^{n-1} \Ker \Ext^i_R(R/J, -) \subseteq \bigcap_{i =0}^{n-1}\Ker H^i(J;-)$$ using a cohomology version of \cite[p. 4, Corollaire~1]{BourAC} and finally show by contradiction that there is no module $M$ such that $M \in  \bigcap_{i =0}^{n}\Ker H^i(J;-) $ and $\Ext^n_R(R/J, M)\neq 0$.
\end{proof}

\begin{rmk} \label{rem:koszul_cos} Let $I$ be a finitely generated ideal and $M$ an $R$-module. Let the following denote the Koszul cochain complex $K^\bullet(I;M)$ defined using an arbitrarily chosen set of generators $\mathbf{x}\coleq x_1, \dots, x_\ell$ of $I$, 

\[ \begin{tikzcd}
0 \arrow[r] &M_{0} \arrow[r, "d^0"] &  M_{1}  \arrow[r, "d^{1}"] & \cdots \arrow[r, "d^{\ell-1}"] &M_\ell  \arrow[r] &0,
\end{tikzcd}
\]
where each $M_i$ is a finite direct sum of copies of $M$. 

For $0 < k \leq \ell$, let $S_{k}(I; M)$  be the module $\Coker (d^{k-1}) $. If $$M\in \bigcap_{i =0}^{k-1}\Ker H^i(I;-)=\bigcap_{i =0}^{k-1} \Ker \Ext^i_R(R/I, -),$$ then the complex provides a finite resolution of length $k$ of $S_{k}(I; M)$ by finite direct sums of copies of $M$. 

In the special case that $M = S$ for a flat ring morphism $R \to S$, then $H_i(I;S)$ is exactly the $i$th Koszul homology of $IS$ with coefficients in $S$ in $\Mod S$, see \cite[Proposition 1.6.7]{BH98}. 

Particular instances of the Koszul-type modules $S_{k}(I; M)$ are going to be our prototype of modules with finite flat dimension.  Their character duals (or also their Matlis duals) are going to be the prototype of pure-injective modules with finite injective dimension, cf. Proposition~\ref{p:torpair-generator}.
\end{rmk}

\subsection{Cotilting classes over commutative rings. }\label{ss:cotiltingcomm}

Over a commutative ring, cotilting classes of cofinite type are classified in terms of certain finite sequences of faithful hereditary torsion pairs of finite type, or equivalently, finite sequences of certain faithful finitely generated Gabriel topologies, see \cite[\S 2]{HS20} for details.

\begin{theorem}\cite[Theorem 6.2]{HS20},\cite[Theorem 4.2, Corollary 4.4]{AHPST14}\label{cotilting_characterisation}
Let $R$ be a commutative ring and $n \geq 0$. Then there are one-to-one correspondences between the following collections. 
\begin{itemize}
	\item[(i)] Descending sequences of torsion-free classes of finite type and closed under injective envelopes
		\[
		(\Ecal_0, \Ecal_1, \dots, \Ecal_{n-1}),
		\]
		such that $\Omega_{-i}(R) \in \Ecal_{i}$ for $0 \leq i <n$,
	\item[(ii)] Descending sequences of finitely generated Gabriel topologies
		\[
		(\Gcal_0, \Gcal_1, \dots, \Gcal_{n-1}),
		\]
		such that $\Ext^i_R(R/J, R)=0$ for every $J \in \Gcal_i$ and $0 \leq i <n$,
	\item[(iii)] $n$-cotilting classes of cofinite type in $\Mod R$,
	\item[(iv)] resolving subcategories of $\mod R$ of modules of projective dimension at most $n$. 
\end{itemize} 
The correspondences (using the notation of \cref{ss:Koszul}) are given as follows.

\begin{itemize}[leftmargin=3cm]
	\item[$(i) \to(ii)$ ] \[\Gcal_i \coleq \{J \leq R \mid \Hom_R(R/J, E)=0 \text{ for every } E \in \Ecal_i \}\]
	\item[$(ii) \to(iii)$ ] \[\Ccal \coleq \bigcap_{i=0}^{n-1}\bigcap_{J \in \Gcal^f_i}  (S_i(J;R)^\top)
	\] where $\Gcal^f_i$ are the finitely generated ideals in $\Gcal_i$. 
\end{itemize}
If $R$ is moreover noetherian, then we can omit the ``of cofinite type'' in $(iii)$ and for every ideal $J$ in $\Gcal_i$,  $\grade {(J;R)}>i$, or $i< \inf\{j \mid \Ext^j_R(R/J,R)\neq 0\}$. 
\end{theorem}

\subsection{Homological properties of commutative noetherian rings.}\label{ss:RG} 

For any prime $\pp \in \Spec R$, we let $k(\pp)$ denote the field of fractions of the residue field $R/\pp$.

In Proposition~\ref{p:torpair-generator} we introduce the \textit{prototype} of module of finite flat dimension  that we will use to prove our classification results. Such modules are a particular instance of the Koszul cosyzygies, cf. Remark~\ref{rem:koszul_cos}; their Matlis dual and their character dual are pure-injective modules of finite injective dimension.

\begin{prop}\label{p:torpair-generator}
Let $R$ be a commutative noetherian ring, and fix a prime $\pp \in \Spec R$ and an integer $k$ such that $\depth{R_\pp}\geq k>0$. Then the following statements hold for $S_{k}(\pp R_\pp; R_\pp)$,
\begin{enumerate}
\item[(1)] $S_{k}(\pp R_\pp; R_\pp)=S_{k}(\pp ; R_\pp)$ is an $R$-module of flat dimension $k$ and an $R_\pp$-module of projective dimension $k$;
\item[(2)] its  character dual $S_{k}(\pp R_\pp; R_\pp)^+$ as well as its Matlis dual $$\mathrm{Hom}_{R _\pp} (S_{k}(\pp R_\pp; R_\pp), E(R_\pp/\pp R_\pp))=S_{k}(\pp R_\pp; R_\pp)^\vee$$  are  both pure-injective modules of injective dimension  $k$ and $$S_{k}(\pp R_\pp; R_\pp)^\top= {}^\perp \left(S_{k}(\pp R_\pp; R_\pp)^+\right)={}^\perp \left(S_{k}(\pp R_\pp; R_\pp)^\vee\right ). $$
\end{enumerate}
Moreover, the following are equivalent for an $R$-module $M$:
\begin{itemize}
	\item[(i)] $M \in S_{k}(\pp R_\pp; R_\pp)^\top$,
	\item[(ii)] $\depth_{R_\pp} (M_\pp) \geq k$,
	\item[(iii)] $\pp \notin \Ass {\bigoplus^{k-1}_{i=0}\Omega_{-i}(M)}$.
\end{itemize}
\end{prop}

\begin{proof} $(1)$ By Proposition~\ref{prop:depth}, if $\depth{R_\pp}\geq k>0$ then the Koszul cochain complex
$K^\bullet(\pp;R_\pp)=K^\bullet(\pp R_\pp;R_\pp)$ is exact up to degree $k-1$ and all its terms are finite direct sums of copies of $R_\pp$, cf. Remark~\ref{rem:koszul_cos}. Therefore $S_{k}(\pp R_\pp; R_\pp)$ is an $R$-module of flat dimension $k$, as claimed in the statement. 

$(2)$ Taking duals with respect to an injective module takes any module to a pure-injective module, and flat modules to injective modules. Moreover, both the character module and the Matlis dual  are duals with respect to an injective cogenerator, hence they are not only exact functors but also preserve dimensions. That is, the dual of the Koszul resolution of $S_{k}(\pp R_\pp; R_\pp)$  witnessing that the module has $R_\pp$-projective dimension $k$, yields an injective resolution  witnessing that the dual module has injective dimension $k$.

For the equalities between the $\Tor$-orthogonals and the left $\Ext$-orthogonals of the Matlis dual of $S_{k}(\pp R_\pp; R_\pp)$ we appeal to Lemma~\ref{tor-induce-cotor} which holds because of the homological formulae given in \S \ref{ss:hom-formulae}. As these homological formulae also hold for the case of the Matlis dual, we also deduce that $S_{k}(\pp R_\pp; R_\pp)^\top= {}^\perp \left(S_{k}(\pp R_\pp; R_\pp)^\vee\right ).$

\medskip

Now we prove  the second part of the statement that describes the $\mathrm{Tor}$-orthogonal of $S_{k}(\pp R_\pp; R_\pp)$.

$(i) \Leftrightarrow (ii)$ Consider the following equivalences. 
\begin{align}
M \in S_{k}(\pp R_\pp; R_\pp)^\top 	&\Leftrightarrow  \Tor^R_{i}(S_{k}(\pp R_\pp; R_\pp), M)=0 \text{ for } 1\leq i \leq k \label{eq1}\\
							&\Leftrightarrow  \Tor^{R_\pp}_{i}(S_{k}(\pp R_\pp; R_\pp), M_\pp)=0 \text{ for } 1\leq i \leq k \label{eq2} \\
							&\Leftrightarrow  \Ext^{i}_{R_\pp}(R_\pp/\pp R_\pp, M_\pp)=0 \text{ for } 0\leq i \leq k-1 \label{eq3}
\end{align}
The equivalence \cref{eq2} holds as $R \to R_\pp$ is a flat ring epimorphism. The equivalence \cref{eq3} holds by \cite[Proposition 5.12]{HS20}, which is shown via the equivalence of the Koszul homology and cohomology, that is that $H_i(I;-) \cong H^{n-i}(I;-)$ for an ideal $I$.
Finally, by definition, $\depth_{R_\pp} M_\pp \geq k$ if and only if $\Ext^{i}_{R_\pp}(R_\pp/\pp R_\pp, M_\pp) =0$ for $0\leq i \leq k-1$.

$(ii) \Leftrightarrow (iii)$
If $k=1$, then $\depth_{R_\pp} (M_\pp) \geq 1$ if and only if $\Hom_{R_\pp}(R_\pp/\pp R_\pp, M_\pp) =0$ if and only if $\pp \notin \Ass M$, as $\pp \in \Ass M$ if and only if $\pp R_\pp \in \Ass M_\pp$.

 For the inductive step, we apply the functor $\Hom_R(R/\pp, -)$ to the minimal injective resolution of $M$. Then, as $\Ass M = \Ass {E(M)}$ and $E(M_\pp) \cong E(M)_\pp$, we find an isomorphism $\Hom_R(R/\pp,\Omega_{-i-1}(M)) \cong \Ext^1_R(R/\pp, \Omega_{-i}(M)) $. By dimension shifting, $\Ext^{i}_{R_\pp}(R_\pp/\pp R_\pp, M_\pp) \cong \Ext^1_R(R/\pp, \Omega_{-i}(M))$ for $i>0$, so the vanishing of the $\Ext^{i}_{R_\pp}(R_\pp/\pp R_\pp, M_\pp)$ for $0 \leq i <k$ coincides with the vanishing of $\Hom_R(R/\pp,\Omega_{-i-1}(M))$, or in other words that $\pp \notin \Ass{\Omega_{-i-1}(M)}$.
\end{proof}

For further quoting, we introduce some of the finitistic dimensions of a ring. These are quite well understood for commutative noetherian rings. 

Let us denote by $\Findim(R) = \sup \{\pd M \mid M \in \Pcal\}$, $\findim(R) = \sup \{\pd M \mid M \in \Pcal^{<\aleph_0}\}$, and $\WFindim(R) = \sup \{\fd M \mid M \in \Fcal\}$, respectively, the \newterm{finitistic}, \newterm{small finitistic}, and \newterm{weak finitistic dimension} of $R$.

 If $R$ is a commutative noetherian ring, by classical results of Bass \cite{B62} and Raynaud-Gruson \cite{RG71}, $\Findim(R)$ coincides with $\Kdim{R}$,  the Krull dimension of $R$, \cite[Théor\`{e}me 3.2.6]{RG71}. Assume now that $\Kdim(R)<\infty$. Then any flat $R$-module belongs to $\Pcal$ \cite[Proposition 6]{Jen70}, \cite[Corollaire 3.2.7]{RG71}, and therefore $\Pcal$ coincides with the class $\Fcal$ of all modules of finite flat dimension. It follows that $\Pcal$ can be described as the class of all modules of flat dimension bounded by $\Kdim(R)$. In symbols, we have $\Pcal_{\Kdim(R)} = \Pcal = \Fcal = \Fcal_{\Kdim(R)}$. 

If $R$ is local with maximal ideal $\mm$ then $\findim(R) = \depth(R)$ by the Auslander-Buchsbaum formula. For any $R$ we have the identity $\WFindim(R) = \sup \{\depth (R_\pp) \mid \pp \in \Spec R\}$, \cite[Proposition 5.1]{B62}.

Finally, the \newterm{finitistic injective dimension} of $R$ is $ \sup \{\id M \mid M \in \Ical\}$, and coincides with $\WFindim(R)$ when $R$ is commutative noetherian by \cite[(8.2)]{Bas60}. In particular, we will often use that, $\Ical = \Ical_{\dim(R) }$.

Let $R$ be a commutative ring, and let $S$ be a multiplicative subset of $R$. For $\Xcal$ a class in $\Mod R$, we let  $\Xcal_S\coleq \{X_S \mid X \in \Xcal\}$, a class in $\Mod R_S$. If the multiplicative subset $R$ is of the form $S=R\setminus \pp$ for $\pp$ a prime ideal of $R$, then we write  $\Xcal_\pp$ instead of $\Xcal_S$.

\begin{prop}\label{C:R_p-cotilting}
Let $R$ be a commutative ring, and let $S$ be a multiplicative subset of $R$. If $(\Ecal, \Ccal)^\top$ is a cotilting $\Tor$-pair, then  $(\Ecal_S, \Ccal_S)^\top$ is a cotilting $\Tor$-pair. If $R$ is, additionally, noetherian, then for every prime $\pp \in \Spec R$, $\Ccal_{\pp}$ is an $n$-cotilting class with  $n\le \depth (R_\pp)$.
\end{prop}

\begin{proof}

By \cref{torpair-localisation} $(\Ecal_S, \Ccal_S)^\top$ is a hereditary $\Tor$-pair and $\Ccal_S$ is a definable class in $\Mod R_S$. Moreover, as $R_S$ is a flat $R$-module, $\Ecal \subseteq \Fcal (R)$ readily implies that $\Ecal _S \subseteq \Fcal (R_S)$. Therefore, $(\Ecal_S, \Ccal_S)^\top$ is a cotilting $\Tor$-pair.

Assume now that $R$ is noetherian.  By the discussion above and the first part of the result,  $\Ecal_\pp \subseteq \Fcal(R_\pp)=\Fcal_{\Kdim(R_\pp)}(R_\pp)$. By \cref{def-tor-pairs}, $\Ccal _\pp$ is an $n$-cotilting class for $n\le \Kdim(R_\pp)$. By \cref{cotilting_characterisation}, $\Ccal _\pp$ is of cofinite type and then, in addition, $n$ is bounded by $\findim(R _\pp)=\depth (R_\pp)$.
\end{proof}

\section{Preliminary results}\label{s:prelimaryresults}

\subsection{Hereditary cotorsion pairs cogenerated by pure-injective modules}\label{ss:induced-cot-pairs-pis-finite-inj}
Let $(\Ccal,\Wcal)$ be a hereditary cotorsion pair cogenerated by a class $\Qcal$ of pure-injective $R$-modules. We define the classes $\Ccal_{(m)} = \Perp{>m}\Wcal = \Perp{>m}\Qcal$. This way we obtain a chain of classes
$$\Ccal = \Ccal_{(0)} \subseteq \Ccal_{(1)} \subseteq \cdots \subseteq \Ccal_{(m)} \subseteq \cdots$$
and for any $m \geq 0$ a hereditary cotorsion pair $(\Ccal_{(m)},\Wcal_{(m)})$. If $(\Ccal,\Wcal)$ is cogenerated by a class $\Qcal$ of pure-injective modules then $(\Ccal_{(m)},\Wcal_{(m)})$ is cogenerated by $\Omega_{-m}\Qcal$, the class of $m$-th minimal cosyzygies of modules from $\Qcal$. Since any such minimal cosyzygy is again pure-injective, see \cref{ss:pure-injective}, this hereditary cotorsion pair is again cogenerated by pure-injectives. If moreover $\Qcal \subseteq \Ical_n(R)$, so all the modules in the right-hand class $\Wcal$ have injective dimension at most $n$, then $(\Wcal_{(m)}, \Ccal_{(m)})^\top = (\Mod R,\Ical_0(R), )$ for all $m\geq n$, and, $\Ccal_{(m)} \subseteq \Ical_{n-m}(R)$ for $0\leq m\leq n$.

\begin{lemma}\label{L:cotorsion-pair-shifts}
	Let $(\Ccal, \Wcal)$ be a hereditary cotorsion pair, and let
    \[\begin{tikzcd}[cramped, sep=small]
	0 \arrow[r] &L \arrow[r] &  M  \arrow[r] & N  \arrow[r] &0
	\end{tikzcd}\]	
    be a short exact sequence of modules in $\lMod R$. Then the following hold. 
	\begin{enumerate}
		\item[(i)] If $L \in \Ccal$ and $M \in \Ccal_{(i)}$ for $i>0$, then $N \in \Ccal_{(i)}$.
		\item[(ii)] If $M \in  \Ccal_{(i)}$ and $N \in \Ccal_{(i+1)}$ for $i\geq 0$, then $L \in \Ccal_{(i)}$.
	\end{enumerate}
	\end{lemma}
   \begin{proof}
For $(i)$, we must show that $\Ext^j_R(N, W)=0$ for every $j> i$ and $W \in \Wcal$. This follows for every $j >i>0$ by the exact sequences
\[
\begin{tikzcd}[cramped, sep=small]
\Ext^{j-1}_R(L, W) \arrow[r] &\Ext^{j}_R(N, W) \arrow[r] &  \Ext^{j}_R(M, W). 
\end{tikzcd}
\]

For $(ii)$ we must show that $\Ext^{j}_R(L, W)=0$ for every $j> i$ and $W \in \Wcal$. There is an exact sequence 
\begin{tikzcd}[cramped, sep=small]
\Ext^{j}_R(M, W) \arrow[r] &\Ext^{j+1}_R(L, W) \arrow[r] &  \Ext^{j+1}_R(N, W) , 
\end{tikzcd}
where the outer two modules vanish for all $j> i$ by assumption. 
\end{proof}

\subsection{Hereditary $\Tor$-pairs}\label{ss:hered-Tor-pairs} A hereditary $\Tor$-pair $(\Ecal, \Ccal)^\top$ induces a class $\Ecal^{\top_{> m}}$ for each integer $m>0$. The class $\Ecal^{\top_{> m}}$ forms the right-hand class of a hereditary $\Tor$-pair, as the associated $\Tor$-pair is generated by $\Upsilon_{m}(\Ecal)$.

  We will let $\Ccal_{(m)}$ denote the right-hand class of the induced $\Tor$-pair, that is $\Ccal_{(m)} \coleq\Ecal^{\top_{> m}} = (\Upsilon_{m}(\Ecal))^\top$ and the hereditary $\Tor$-pair by $(\Ecal_{(m)}, \Ccal_{(m)})^\top$. In particular, $\Ccal_{(0)}= \Ccal \subseteq \Ccal_{(1)} \subseteq \cdots \subseteq \Ccal_{(i)} \subseteq \cdots $ and $\Ecal_{(0)}= \Ecal \supseteq \Ecal_{(1)} \supseteq \cdots \supseteq \Ecal_{(i)} \supseteq \cdots $. If moreover $\Ecal \subseteq \Fcal_n(R)$, so all the modules in the left-hand class have flat dimension at most $n$, then $(\Ecal_{(m)}, \Ccal_{(m)})^\top = (\Fcal_0(R), \lMod R)^\top$ for all $m\geq n$, and, $\Ecal_{(m)} \subseteq \Fcal_{n-m}(R)$ for $0\leq m\leq n$.

\begin{lemma}\label{L:Tor-pair-shifts}
Let $(\Ecal, \Ccal)^\top$ be a hereditary $\Tor$-pair, and let
\[\begin{tikzcd}[cramped, sep=small]
0 \arrow[r] &L \arrow[r] &  M  \arrow[r] & N  \arrow[r] &0
\end{tikzcd}\]
be a short exact sequence of modules in $\lMod R$. Then the following hold. 
\begin{enumerate}
	\item[(i)] If $L \in \Ccal$ and $M \in \Ccal_{(i)}$ for $i>0$, then $N \in \Ccal_{(i)}$.
	\item[(ii)] If $M \in  \Ccal_{(i)}$ and $N \in \Ccal_{(i+1)}$ for $i\geq 0$, then $L \in \Ccal_{(i)}$.
\end{enumerate}
\end{lemma}
\begin{proof}
Analogous to the proof of \cref{L:cotorsion-pair-shifts}.
\end{proof}

\subsection{Cotorsion pairs cogenerated by pure-injective modules and colocalisation over commutative rings} 
Let $R$ be a commutative ring, let $S$ be a multiplicative subset of $R$, and let $R_S$ be the localisation of $R$ at $S$. 
Localisation over an arbitrary commutative ring preserves flat dimension, and interacts well with the $\Tor_R$ functor. However, in general, localisation does not preserve the injective dimension of a module. Instead, the \newterm{colocalisation} of an $R$-module $M$, which is the module $\Hom_R(R_S, M)$, has many desirable properties when $M$ is pure-injective. 
More generally, let $R \to R'$ be any ring homomorphism. Then for an $R$-module $N$ we have following basic properties:
	\begin{itemize} 
		\item $\id_R(N) = n \implies \id_{R'}\Hom_R(R',N) \leq n$ provided that $\Ext_R^j(R',\Omega_{-i}N) = 0$ for all $i\geq0$ and all $j>0$.
		\item $N$ pure-injective $\implies$ $\Hom_R(R',N)$ pure-injective as an $R'$-module. 
        \item If $N$ is an $R'$-module and $\Tor^R_j(M, R')=0$ for every $j>0$, then $\Ext^i_R(M,N)\cong \Ext^i_{R'}(M\otimes_RR',N)$ for every $i$.
	\end{itemize}

In particular, if $R'=R_S$, all the above assumptions are satisfied for the case that $N$ is a pure-injective module, see \cref{ss:pure-injective}.
We will denote the colocalisation of a class $\Xcal$ with respect to a multiplicative subset $S$ by  $\Xcal^S\coleq \{X^S \mid X \in \Xcal\}$, where $X^S = \Hom_R(R_S,X) \in \Mod{R_S}$. If $S = R \setminus \pp$ for a prime ideal $\pp$, we will write $X^\pp$ instead of $X^S$ and $\Xcal^\pp$ instead of $\Xcal^S$.

Suppose $R$ is a commutative ring, $M$ is any module, $F$ is a flat $R$-module and $N$ is a pure-injective module. Uniting the homological identity in \cref{ext-flat-pi} for pure-injective modules with formula \cref{eq:derived-tensor-hom} in \cref{ss:hom-formulae}, we deduce the following isomorphisms. 

\begin{equation}\label{eq:the-eq-general}
\Hom_R(F, \Ext_R^i(M, N))  \cong\Ext_R^i(F \otimes_RM, N) \cong \Ext_R^i(M, \Hom_R(F, N))
\end{equation}

In particular, the equation \cref{eq:the-eq-general} shows that colocalisation of pure-injective modules commutes with the $\Ext$ functor in the second variable. Thus if $F$ is the localisation ring $R_S$, combining the above information we have the following isomorphisms for a pure-injective module $N$, $M$ any $R$-module, and $i \geq 0$. 

\begin{equation}\label{eq:the-eq}
\Ext_R^i(M, N)^S\cong\Ext_R^i(M_S, N) \cong  \Ext_R^i(M, N^S) \cong \Ext_{R_S}^i(M_S, N^S)
\end{equation}

    \begin{lemma}\label{cotorpair-localisation}
	Let $R$ be a commutative ring and $S$ a multiplicative subset of $R$. Consider a hereditary cotorsion pair $(\Ccal, \Wcal)$ cogenerated by a set $\Qcal$ of pure-injective modules. Then there is a hereditary cotorsion pair $(\Ccal_S, \Vcal)$ in $\Mod R_S$ cogenerated by the set $\Qcal^S$ of pure-injective modules, and moreover $\Ccal_S = \Ccal \cap \Mod R_S$.
	
	In addition $(\Ccal_S)_{(i)}$ coincides with $(\Ccal_{(i)})_S$ for $i \geq 0$.
	\end{lemma}
	\begin{proof}
	First note that $\Ccal_S = \Ccal \cap \Mod R_S$ as $\Ccal$ is closed under direct limits. 
	
	Take $P \in \Mod R$ a pure-injective module, and $M \in \Mod R$. Then for any $i>0$: $$\Ext_R^i(C_S,P) \cong \Ext^i_{R_S}(C_S, P^S) \cong \Ext^i_{R}(C,P^S)$$ by \cref{eq:the-eq}.  Then $M_S \in \Ccal_S$ if and only if $M_S \in \Perp{>0}\Qcal^S$.
	\end{proof}

\subsection{Localisations and $\Tor$-pairs}

For further use we record the good behaviour of $\Tor$-pairs with respect to localisation. Recall the definitions before \cref{C:R_p-cotilting}.

\begin{lemma}\label{torpair-localisation}
Let $R$ be a commutative ring, and let $S$ be a multiplicative subset of $R$. Consider a $\Tor$-pair $(\Ecal, \Ccal)^\top$. Then there is a $\Tor$-pair $(\Ecal_S, \Ccal_S)^\top$ in $\Mod R_S$, and moreover $\Ecal_S = \Ecal \cap \Mod R_S$ and $\Ccal_S = \Ccal \cap \Mod R_S$.

If in addition the $\Tor$-pair is hereditary, then $((\Ecal_S)_{(i)}, (\Ccal_S)_{(i)})^\top$ coincides with $((\Ecal_{(i)})_S, (\Ccal_{(i)})_S)^\top$ as $\Tor$-pairs in $\Mod R_S$ for $i \geq 0$.
\end{lemma}
\begin{proof}
First note that $\Ecal_S = \Ecal \cap \Mod R_S$ and $\Ccal_S = \Ccal \cap \Mod R_S$ as classes in a $\Tor$-pair are always closed under direct limits. 

Take $M\cong M_S \in \Mod R_S$, $E \in \Ecal$. Then $\Tor_1^{R_S}(E\otimes_RR_S, M) \cong \Tor_1^{R}(E,M)$ as $R \to R_S$ is a flat ring morphism.  The latter module vanishes for every $E \in \Ecal$ if and only if $M \in \Ccal$, so the right $\Tor$-orthogonal of $\Ecal_S$ in $\Mod R_S$ is exactly $\Ccal_S$, and a symmetric argument shows that the left $\Tor$-orthogonal of $\Ccal_S$ in $\Mod R_S$. 
\end{proof}

\subsection{Pure-injective modules over commutative noetherian rings}

When $R$ is a commutative noetherian ring, it is known that every flat cotorsion module is pure-injective \cite[Lemma 3.2.3]{Xu96}. Moreover, by work of Enochs \cite{Eno84}, or \cite[Section 5.3, Theorem 5.3.28]{EJ11}, every flat cotorsion module is isomorphic to a product $\prod_{\pp \in \Spec R} T_\pp$, where $T_\pp$ is the $\pp R_\pp$-completion of a free $R_\pp$-module.

\begin{lemma}\label{R_m-gen-cot}
    Let $R$ be a commutative noetherian ring and $C \in \Mod R$ a cotorsion module. If $\Hom_R(R_\mm, C)=0$ for every maximal ideal $\mm \in \Spec R$, then $C =0$. 
\end{lemma}
\begin{proof}
    Let $F \overset{\phi}\to C \to 0$ be a flat cover of $C$. Using \cite[Theorem 5.3.28]{EJ11}, we can explicitly write $F$ as $\prod_{\pp \in \Spec R} T_\pp$ as described above.
    
    First note that if $\mm$ is a maximal ideal of $R$ and $\Hom_R(R_\mm, C)=0$, then necessarily $\Hom_R(M_\mm, C)=0$ for every $R_\mm$-module $M_\mm$. Therefore, as $T_\pp$ is an $R_\mm$-module for every maximal ideal $\mm$ containing $\pp$, $\Hom_R(T_\pp, C)=0$ for every $\pp \in \Spec R$. As the contravariant $\Hom$ functor commutes with coproducts, $\Hom_R(\bigoplus_{\pp \in \Spec R} T_\pp, C) \cong \prod_{\pp \in \Spec R}\Hom_R( T_\pp, C)=0$. Let 
    \[
        \begin{tikzcd}
           0 \arrow[r] &  \displaystyle{\bigoplus_{\pp \in \Spec R} T_\pp} \ar[r, "\tau"] &
           \displaystyle{\prod_{\pp \in \Spec R} T_\pp} \ar[r, "\pi"] &
           \Coker (\tau) \ar[r] &
           0
        \end{tikzcd}
    \]
    denote the short exact sequence where $\tau$ is the natural inclusion of a coproduct into a product. As $\phi\tau =0$, $\phi$ factors through $\Coker (\tau)$, that is, there exists a homomorphism $\beta \colon \Coker( \tau) \to C$ such that $\beta \pi = \phi$. 
    
    As $\tau$ is a pure monomorphism, it follows that $\Coker (\tau)$ is a flat $R$-module. Thus, by the flat-precover property of $\phi$, there exists a morphism $\alpha \colon \Coker (\tau) \to \prod_{\pp \in \Spec R} T_\pp$ such that $\phi \alpha = \beta$. Therefore, $\phi = \beta \pi = \phi \alpha \pi$, and by the minimality property of the flat cover $\phi$, $\alpha \pi$ is an automorphism. We conclude that $\pi$ is an isomorphism, and therefore $F \cong \prod_{\pp \in \Spec R} T_\pp=0$, so also $C=0$, as desired.  
\end{proof}

\begin{rmk}\label{rmk:counterexample-Rm-generating}
    The noetherianity assumption is required in \cref{R_m-gen-cot}. Namely, by work of Trlifaj \cite{Trl95} there exists a superdecomposable pure-injective module $N$ over a commutative von Neumann regular ring which is not semi-artinian, and by work of Stevenson \cite{Ste14}, $\Hom_R(R_\mm, N)=0$ for every maximal ideal $\mm$. 
\end{rmk}

\begin{prop}\label{coloc-orth-pi}
	 Let $R$ be a commutative noetherian ring. Let $N$ be a pure-injective module and $i\geq0$ an integer. Then, for an $R$-module $M$, $\Ext^i_R(M, N)=0$ if and only if $\Ext^i_R(M, N^\mm)=0$ for every maximal ideal $\mm \in \Spec R$. 
\end{prop}
\begin{proof}
	Let $N$ be a pure-injective module and $\mm$ a maximal ideal. By \cref{eq:the-eq}, there is an isomorphism

 \[
\Ext_R^i(M,N^\mm) \cong \Ext_R^i(M,N)^\mm.
 \]
 
 Thus, it is sufficient to show that $\Ext_R^i(M,N)^\mm$ vanishes for every maximal ideal $\mm \in \Spec R$ if and only if $\Ext_R^i(M,N)$ vanishes.
 
The ``only if'' statement is trivial. For the ``if'' statement, recall that $\Ext_R^i(M,N)$ is a pure-injective $R$-module by \cite[Proposition 7]{War69}, and apply \cref{R_m-gen-cot}.
 \end{proof}

\begin{rmk}
    The authors do not know whether \cref{coloc-orth-pi} holds without the noetherianity assumption. Certainly, if $R$ is commutative and $N$ is a pure-injective $R$-module, then we always have that $ {}^\perp N\subseteq  {}^\perp \prod_\mm N^\mm$ which ranges over the maximal ideals $\mm$. However, the proof in \cref{coloc-orth-pi} for the converse inclusion uses \cref{R_m-gen-cot}, which is certainly not true outside noetherian rings as noted in \cref{rmk:counterexample-Rm-generating}. However it would not be possible to find a counterexample to \cref{coloc-orth-pi} over a von Neumann regular ring as every pure-injective module is injective in this setting. 

    We also remark that \cref{coloc-orth-pi} is clear for $\Tor$-pairs. In fact, by the central isomorphism in \cref{eq:the-eq}, if  ${}^\perp N=  {}^\perp \prod_\mm N^\mm$ then $M \in {}^\perp N$ if and only if $M_\mm \in{}^\perp N$ for every maximal ideal $\mm$, a fact which is straightforward to deduce for classes in Tor pairs. Furthermore, if there exists a commutative ring $R$ with a pure-injective $R$-module $N$ such that ${}^\perp N \subsetneq  {}^\perp \prod_\mm N^\mm$, then ${}^\perp N$ could not be a class in a $\Tor$-pair. 
\end{rmk}

We have the following corollary, the special case of $N$ being a cotilting module appears in \cite[Theorem 2.4]{TS14}.
\begin{cor}\label{c:loc-cotor-pair}
Let $R$ be a commutative noetherian ring. Let $N$ be a pure-injective module. Then, ${}^\perp N = {}^\perp (\prod_{\mm} N^\mm)$. Moreover, if ${}^\perp N$ is a definable class in $\Mod R$ then ${}^\perp (N^\mm)$ is a definable class in $\Mod R_\mm$ for each maximal ideal $\mm$. 
\end{cor}

\begin{proof}
The first statement follows directly from \cref{coloc-orth-pi}. The moreover clause follows as ${}^\perp N^\mm = ({}^\perp N)_\mm$ by \cref{cotorpair-localisation}, and as localisations of definable classes are definable. 
\end{proof}

\section{The main classification}\label{s:hered-tor-pair}

Throughout this section, $R$ will denote a commutative noetherian ring. Recall $\Ical$ denotes the class of all $R$-module of finite injective dimension and $\Pcal\Ical_0$ is the class of all pure-injective $R$-modules. The goal of this section is to prove our main classification result \cref{T:Torpair-characterisation}. We start by showing that the left constituents of the hereditary cotorsion pairs cogenerated by modules from $\Ical \cap \Pcal\Ical_0$ are closed under taking injective envelopes. 

\begin{lemma} \label{l:associatedprime}
Let $R$ be a local ring with maximal ideal $\mm$, and take a short exact sequence $ \begin{tikzcd}[cramped, sep=small]
0 \arrow[r] &L \arrow[r] &  M  \arrow[r] &N  \arrow[r] &0,
\end{tikzcd}
$
such that $L$ is $\mm$-torsion, or equivalently, $L$ is filtered by copies of $R/\mm$. Then $\Ass M = \Ass L \cup \Ass N$ 
\end{lemma}

\begin{proof}
Suppose that $L \neq 0$. It is sufficient to show that if $\pp \in \Ass N \setminus \{\mm\}$, then $\pp \in \Ass M$ for $\pp \in \Spec R$. Fix $\pp \in \Ass N$ and localise the short exact sequence at $\pp$. Then $L_\pp = 0$ as $(R/\mm)_\pp=0$, so as the sequence remains exact $M_\pp \cong N_\pp$, so $\pp R_\pp \in \Ass {N_\pp}$ implies that $\pp \in \Ass M$, as required.\end{proof}

\begin{lemma} \label{l:injective-hull}
Let $(\Ccal, \Wcal)$ be a hereditary cotorsion pair cogenerated by a class $\Qcal \subseteq \Pcal\Ical_0$. If $k(\pp) \in \Ccal$ for some $\pp \in \Spec R$, then the modules $E(k(\pp)) , \Omega_{-1}(k(\pp)) \in \Ccal$. More generally, for an $R$-module $M$, if $\{k(\pp) \mid \pp \in \Ass M\} \subseteq \Ccal$, then $E(M) \in \Ccal$. 
\end{lemma}

\begin{proof}
The first statement holds since $E(k(\pp))$ and $\Omega_{-1}(k(\pp))$ are filtered by direct sums copies of the module $k(\pp)$, and since $\Ccal$ is closed under direct limits. The second statement holds since for a module $M$, $E(M) \cong \bigoplus_{\pp \in \Ass M} E(k(\pp))^{(\alpha_\pp)}$ for cardinals $\alpha_\pp$. So $E(M)\in \Ccal$ by the first statement. 
\end{proof}

\begin{lemma}\label{l:associated-residue}
Let $(\Ccal, \Wcal)$ be a hereditary cotorsion pair cogenerated by a class $\Qcal \subseteq \Ical \cap \Pcal\Ical_0$. For any $M \in \Ccal$, if $\pp \in \Ass M$, then $k(\pp) \in \Ccal$. 
\end{lemma}
\begin{proof}
It is enough to prove the statement for $\Qcal = \{N\}$ where $N$ is pure-injective of injective dimension $n > 0$, we proceed by induction on $n$. For $n =1$, $\Ccal$ is closed under submodules. So if $\pp \in \Ass M$, then $R/\pp \in \Ccal$. Moreover, $\Ccal$ is closed under direct limits, and $k(\pp)$ is a direct limit of copies of $R/\pp$ in $\Mod R$, so it follows that $k(\pp) \in \Ccal$. 

Fix an $n>1$, and suppose that the statement holds for all $0 <k <n$.  In particular, this means that the statement holds for the cotorsion pairs $(\Ccal_{(n-k)}, \Wcal_{(n-k)})$,  where $(\Ccal_{(n-k)}, \Wcal_{(n-k)})$ is cogenerated by a pure-injective module of injective dimension $k$, see \cref{ss:induced-cot-pairs-pis-finite-inj}. Take $M \in \Ccal$ and $\pp \in \Ass M$. We claim that we can assume without loss of generality that $M = M_\pp$, so $\pp$ is maximal in $\Ass M$. This is because, by \cref{cotorpair-localisation}, $M_\pp \in \Ccal$, and, moreover, $\pp \in \Ass M$ if and only if $\pp R_\pp \in \Ass M_\pp \subseteq \Spec {R_\pp}$.

Suppose $L$ is a $\pp$-torsion $R_\pp$-submodule of $M \cong M_\pp$, that is, $L$ is filtered by copies of $k(\pp)$. By the induction hypothesis, we can assume that $k(\qq) \in \Ccal_{(1)}$ for every $\qq \in \Ass M$ and, hence, also $L \in \Ccal_{(1)}$. Similarly, by \cref{l:injective-hull}, for any module $X$ such that $\Ass X \subseteq \Ass M$, its injective envelope $E(X)$, is in $\Ccal_{(1)}$. Consider the following commutative diagram with exact rows and columns formed by the horseshoe lemma.

\begin{equation}\label{eq:p-torsion-seq}
\begin{tikzcd}
& 0 \arrow[d] & 0 \arrow[d] & 0 \arrow[d]\\
0 \arrow[r] &L \arrow[r] \arrow[d] &  M  \arrow[r] \arrow[d] &  N  \arrow[r] \arrow[d] &0\\
0 \arrow[r] &E(L) \arrow[r] \arrow[d] &  E(L) \oplus E(N)  \arrow[r] \arrow[d] &  E(N)  \arrow[r] \arrow[d] &0\\
0 \arrow[r] &\Omega_{-1}(L) \arrow[r]\arrow[d]  &  Y  \arrow[r]\arrow[d]  &  \Omega_{-1}(N)  \arrow[r]\arrow[d]  &0\\
& 0 & 0 & 0
\end{tikzcd}
\end{equation}

By the above explanation and \cref{l:associatedprime}, $E(L)$ and $ E(N)$ are in $ \Ccal_{(1)}$, so also their direct sum is in $\Ccal_{(1)}$. Considering the middle vertical column and applying \cref{L:cotorsion-pair-shifts}(i), $Y \in \Ccal_{(1)}$. As also $\Omega_{-1}(L)$ is filtered by copies of $k(\pp)$, $\Omega_{-1}(L) \in \Ccal_{(1)}$, so it follows that $\Omega_{-1}(N) \in \Ccal_{(2)}$ by considering the last row. Therefore, by applying \cref{L:cotorsion-pair-shifts}(ii) to the rightmost column, we find that $N \in \Ccal_{(1)}$. In particular, $L \in \Ccal$ by applying \cref{L:cotorsion-pair-shifts}(ii) to the first row. 

As this proof holds for any $\pp$-torsion $R_\pp$-submodule $L$ of $M$, it holds for $k(\pp)$, so we conclude that $k(\pp) \in \Ccal$, as required. \end{proof}

Now the combination of \cref{l:injective-hull} and of \cref{l:associated-residue} immediately yield the claimed closure by injective hulls.

\begin{cor}\label{C:C-inj-env}
Let $(\Ccal, \Wcal)$ be a hereditary cotorsion pair cogenerated by a class $\Qcal \subseteq \Ical \cap \Pcal\Ical_0$. Then $\Ccal$ is closed under injective envelopes. 
\end{cor}

The embedding of $L_{\Tor_1}$ into $L_{\Ext^1}$ of \cref{ss:lattice-of-pairs} takes hereditary $\Tor$-pairs generated by a class of modules in $\Fcal$ to hereditary cotorsion pairs cogenerated by modules from $\Ical \cap \Pcal\Ical_0$. Thus, we also get:

\begin{cor}\label{C:TP-inj-env}
Let $(\Ecal, \Ccal)^\top$ be a hereditary $\Tor$-pair generated by a class $\Gcal \subseteq \Fcal$. Then $\Ccal$ is closed under injective envelopes. 
\end{cor}

Before proceeding, we show that constituents of a $\Tor$-pair not satisfying the assumptions of \cref{C:TP-inj-env} may fail to be closed under injective envelopes. We take the opportunity to exhibit some examples.

\begin{ex}\label{ex1} 
Let $(R, \mm, k)$ be a local noetherian which is not self-injective and with $\mm  \in \Ass R$. Consider the $\Tor$-pair generated by $k$. Then $R \in k^\top$, however we claim that $E(R) \notin k^\top$. It suffices to show that $E(k) \notin k^\top$ as $E(k)$ is a direct summand of $E(R)$ by the assumption that $\mm  \in \Ass R$. 
Recall that $\Hom_R(E(k), E(k)) \cong \hat{R}$, the $\mm$-adic completion of $R$.

Suppose $d = \dim R>0$, so in particular by Grothendieck's vanishing theorem, the local cohomologies $H_{\hat{\mm}}^d(\hat{R}) \cong H_\mm^d(R) \neq 0$, see \cite[Lemma~3.5.4 (d),Theorem~3.5.7]{BH98}.
Assume, by way of contradiction, that $\Tor^R_d(E(k), k)=0$. Then, also $\Ext^d_R(k, \hat{R})$ vanishes via the isomorphism 
\[\Ext^d_R(k, \hat{R})  \cong \Hom_R(\Tor^R_d(E(k), k), E(k)) =0 \] given by the formula \cref{eq:derived-tensor-hom}. So by \cite[Remarks~3.5.3(c)]{BH98} 
\[
H_{\hat{\mm}}^d(\hat{R}) \cong  \varinjlim_i \Ext_{\hat{R}}^d(\hat{R}/\hat{\mm}^i, \hat{R}) \cong \varinjlim_i \Ext_R^d(R/\mm^i, \hat{R})=0,\]
 a contradiction.

\end{ex}

The following remark observes that over commutative rings, the classes of a $\Tor$-pair are interchangeable, but not symmetric. Thus, a $\Tor$-pair generated by a module of infinite flat dimension may give rise to a hereditary $\Tor$-pair (co)generated by modules of finite flat dimension. See \cref{hered-tor-pair-reg-ring} for when hereditary $\Tor$-pairs are both generated and cogenerated by modules of bounded flat dimension.

\begin{rmk}
Given a generating set of a $\Tor$-pair, it is not immediately clear whether it will be generated or cogenerated by modules of bounded flat dimension. For example, let $(R, \mm, k)$ be a local noetherian self-injective ring which is not a field, and consider again the $\Tor$-pair generated by $k$. Then $k^{\top_1}= k^\top= \Pcal_0 = \Ical_0 \subsetneqq \Mod R$, so the $\Tor$-pair generated by $k$ is $(\Mod R, \Pcal_0)^\top$. The class $\Pcal_0$ is closed under injective envelopes, but the flat dimension of $k$ is not bounded as $R$ is not a field. 
However, clearly this $\Tor$-pair is cogenerated by $R$.
\end{rmk}

There exist hereditary $\Tor$-pairs which are not generated (or cogenerated) by classes of modules of finite flat dimension, as seen in the following example.

\begin{ex}\label{ex2}
Let $R$ be a non-Gorenstein commutative noetherian ring of finite Krull dimension with a Gorenstein flat module which is not flat (the commutative noetherian rings which do not have an analogue of the Govorov-Lazard Theorem for the Gorenstein flat modules provide examples of such rings, see \cite[Theorem 2.7 and Example 2.8]{HJ11}).
Let $\Gcal \Fcal_0$ denote the Gorenstein flat modules in $\Mod R$. Then, by \cite[Corollary 5.7]{SS20}, $(\Gcal \Fcal_0, \Gcal \Fcal_0^\top)^\top$ forms a hereditary $\Tor$-pair. By assumption, $\Fcal_0 \subsetneq \Gcal \Fcal_0$, therefore there exists a module in $\Gcal \Fcal_0$ which is of infinite flat dimension by \cite[Theorem 3.19]{H04}. Moreover, $\Ical \subseteq \Gcal \Fcal_0^\top$ where $\Ical$ denotes the class of modules of finite injective dimension by \cite[Theorem 3.14]{H04}. By the assumption that $R$ is not Gorenstein, there exists an injective module of infinite flat dimension by \cite[Proposition 1]{I80}. Thus, $(\Gcal \Fcal_0, \Gcal \Fcal_0^\top)^\top$ is a hereditary $\Tor$-pair which is not generated by modules of bounded flat dimension on either side, as there is a uniform bound on flat dimensions as $R$ is of finite dimension. Moreover, $\Gcal \Fcal_0^\top$ is closed under injective envelopes as it contains all injective modules. 

\end{ex}

Let us continue again in the direction of the promised classification. The closure under injective envelopes yields the following reduction, which is important for one of the main inductive arguments in \cref{T:Torpair-characterisation}.

\begin{lemma}\label{C-induction-step}
Let $(\Ccal, \Wcal)$ be a hereditary cotorsion pair cogenerated by a class $\Qcal \subseteq \Ical \cap \Pcal\Ical_0$. Then the following hold for any $R$-module $M$. 
\begin{itemize}
	\item[(i)] $M \in \Sub \Ccal$ if and only if $E(M) \in \Ccal$.
	\item[(ii)]$M \in \Ccal$ if and only if $M \in \Sub \Ccal$ and $\Omega_{-1}(M) \in \Ccal_{(1)}$. 
\end{itemize}

\end{lemma}

\begin{proof}
For $(i)$, if $M \in \Sub \Ccal$, then $E(M)$ is a direct summand of a module in $\Ical_0(R) \cap \Ccal$ by properties of injective envelopes and because $\Ccal$ is closed under injective envelopes by \cref{C:C-inj-env}. Hence,  $E(M)\in \Ccal$. If  $E(M) \in \Ccal$, then trivially  $M \in \Sub \Ccal$.

For $(ii)$, suppose $M \in \Ccal$. Then it follows trivially that  $M \in \Sub \Ccal$, and $\Omega_{-1}(M) \in \Ccal_{(1)}$ follows since $E(M) \in \Ccal$ by \cref{C:C-inj-env} and applying \cref{L:cotorsion-pair-shifts}$(i)$. 

Now suppose that $M \in \Sub \Ccal$ and $\Omega_{-1}(M) \in \Ccal_{(1)}$. Then $E(M) \in \Ccal$ by $(i)$, so by \cref{L:cotorsion-pair-shifts}(ii), $M \in \Ccal$. 
\end{proof}

The following lemma will be used in the proof of \cref{boundedbydepth}.  
\begin{lemma}\label{Torffddepth}
    Let $(R,\mm,k)$ be a local commutative noetherian ring and $N$ a pure-injective module of finite injective dimension. Then $\Ext_R^j(k(\pp),N)=0$ for every $j>\depth (R_\pp)$.
\end{lemma}
\begin{proof}
    In view of the isomorphism $\Ext_R^j(k(\pp),N)\cong \Ext_{R_\pp}^j(k(\pp),N^\pp)$ (see \cref{eq:the-eq}), it is sufficient to prove the statement when $\pp$ is the maximal ideal.

Let $x_1, \dots, x_m$ denote a maximal regular sequence, so $m = \depth (R)$, and let $J \coleq \langle x_1, \dots, x_m \rangle$. Fix a module $N$ of finite injective dimension over $R$. Then, for $j> \depth (R)$, $\Ext_R^j(R/J,N)=0$ as $\pd (R/J) =\depth (R)$, so in particular, $\Ext^{i}_R(R/J,\Omega_{-\depth (R)}N)=0$ for all $i>0$. 

As $J$ is generated by a maximal regular sequence, $R/J$ is a local ring with zero depth, $\mm$ is an associated prime of $R/J$. The module $\Omega_{-\depth (R)}N$ is pure-injective of finite injective dimension, so by \cref{C:C-inj-env}, $E(R/J) \in {}^\perp \Omega_{-\depth (R)}N$. As $\mm$ is an associated prime of $R/J$, $E(k)$ is a direct summand of $E(R/J)$, and by \cref{l:associated-residue}, it follows that $k \in {}^\perp \Omega_{-\depth (R)}N$, so $\Ext^i_R(k, N)=0$ for all $i>\depth R$ as desired.

\end{proof}

The following two propositions provide the two main assignments of the classification \cref{T:Torpair-characterisation}. First, we introduce the following function, which will play an essential role in our classification results.

\begin{dfn} \label{def:functiondepth} We let the \newterm{depth function} be  the function $\depth \colon \Spec R \to \Zbb_{\geq 0}$ assigning to each prime ideal $\pp$, the depth of the ring localised at $\pp$, i.e.:
    $$\depth(\pp) \coloneqq \depth (R_\pp) = \inf \{ i \in \Zbb_{\geq 0} \mid \Ext_{R_\pp}^i(k(\pp),R_\pp)\neq 0\}.$$ 
\end{dfn} 

We consider functions on the spectrum $\Spec R$ with values in non-negative integers $\Zbb_{\geq 0}$.  For a general function $\phi\colon \Spec R \to \Zbb_{\geq 0}$, we will write $\phi \leq \depth$ as a shorthand notation for $\phi(\pp) \leq \depth(\pp)$ for all $\pp \in \Spec R$.

\begin{prop}\label{boundedbydepth}
Let $(\Ccal, \Wcal)$ be a hereditary cotorsion pair cogenerated by a class $\Qcal \subseteq \Ical \cap \Pcal\Ical_0$. Define a map $\phi \colon \Spec R \to \Zbb_{\geq 0}$ by $\pp \mapsto \inf \{i \in \Zbb_{\geq 0} \mid k(\pp) \in \Ccal_{(i)}\} $. Then $\phi \leq \depth$. 
\end{prop}

\begin{proof}
We have for any $\pp \in \Spec R$
$$\Ext^i_{R}(k(\pp),N) \cong \Ext^i_{R_\pp}(k(\pp),N^\pp)$$
for any pure-injective $R$-module $N$ and $i>0$. By  \cref{Torffddepth}, we know that $\Ext^j_{R_\pp}(k(\pp),N^\pp)=0$ for every pure-injective $N$ of finite injective dimension and every $j > \depth(\pp)$. It follows that $k(\pp) \in \Perp{>\depth(\pp)}\Qcal$, which by definition means that $k(\pp) \in \Ccal_{(\depth(\pp))}$ so $\phi(\pp) \leq \depth(\pp)$, as required. 
\end{proof}

\begin{prop}\label{function-to-torpair}
Let $\phi \colon \Spec R \to \Zbb_{\geq 0}$ be a function satisfying $\phi \leq \depth$. Then there is a hereditary $\Tor$-pair $(\Ecal, \Ccal)^\top$ generated by a set of modules of finite flat dimension such that $\Ccal = \{M \in \Mod R \mid \depth_{R_\pp}(M_\pp) \geq \phi(\pp) ~\forall \pp \in \Spec R\}$.

Moreover, $\Ccal_{(i)} = \{M \in \Mod R \mid \depth_{R_\pp}(M_\pp) \geq (\phi(\pp)-i) ~\forall \pp \in \Spec R\}$ for any $i \geq 0$. In particular, $k(\pp) \in \Ccal_{(i)}$ if and only if $\phi(\pp)\leq i$.
\end{prop}

\begin{proof}
We construct a generating set $\Scal_\phi$ for the $\Tor$-pair, consisting of modules of finite flat dimension in $\Mod R$. For each $\pp \in \Spec R$ we have $\phi(\pp) \leq \depth(\pp)$, so \cref{p:torpair-generator} yields that the module $S_{\phi(\pp)}(\pp R_\pp; R_\pp)$ constructed in \cref{ss:Koszul} is of finite flat dimension over $R$ for each $\pp \in \Spec R$. We let 
\[
\Scal_\phi \coleq \{S_{\phi(\pp)}(\pp R_\pp; R_\pp)\}_{\pp \in \Spec R},
\] 
and let $(\Ecal, \Ccal)^\top$ be the $\Tor$-pair generated by $\Scal_\phi$. Then, again by \cref{p:torpair-generator}, we conclude that the modules in $\Ccal_\phi \coleq \Scal_\phi^\top$ are of the desired form.

For the moreover statement, let $M \in \Mod R$ and consider a short exact sequence $0 \to \Omega_1 M \to P \to M \to 0$ with $P$ projective. Then $\depth_{R_\pp}((\Omega_1 M)_\pp) = \min(\depth_{R_\pp}(M_\pp)+1,\depth(\pp))$. Since $P \in \Ccal$, we get $M \in \Ccal_{(1)} \iff \Omega_1 M \in \Ccal$ by \cref{L:cotorsion-pair-shifts}, which yields the desired description of $\Ccal_{(1)}$, then we proceed by induction. The final claim then follows directly as $\depth_{R_\pp}(k(\pp))=0$.
\end{proof}

\begin{rmk}
In the generating set $\Scal_\phi$ constructed above, for $n>0$ we note that 
\[\Scal_\phi \cap \Fcal_n(R) = \{S_{\phi(\pp)}(\pp R_\pp;R_\pp) \mid \phi(\pp)\leq n\}.
\]
However, for $i \geq 0$, the induced $\Tor$-pairs are generated by the following set of $R$-modules of finite flat dimension,   
\[
(\Ecal_\phi)^{\top_{> i}}=(\Ccal_\phi)_{(i)}= \{S_{\phi(\pp)-i}(\pp R_\pp;R_\pp) \mid \phi(\pp)>i\}^\top.
\]
\end{rmk}
In the case that the flat dimensions of the generators are bounded by 1, the class $\Ccal$ is closed under submodules. Therefore, we can readily capture these classes in terms of subsets of $\Spec R$; this will serve as the base step of the coming induction.
For a subset $X$ of $\Spec R$, we will denote by $\Ccal_X$ the class of $R$-modules $\{M \in \Mod R \mid \Ass M \subseteq X\}$.  

\begin{lemma}\label{L:CX}
    Let $\Ccal$ be a class of $R$-modules closed under submodules, extensions, and direct limits. Then there is a subset $X$ of $\Spec R$ such that $\Ccal = \{M \in \Mod R \mid \Ass M \subseteq X\}$.
\end{lemma}
\begin{proof}
    Given $M \in \Ccal$, we have $k(\pp) \in \Ccal$ for any $\pp \in \Ass M$ as $\Ccal$ is closed under submodules and direct limits. Then also $E(k(\pp)) \in \Ccal$, because $E(k(\pp))$ admits a filtration by coproducts of copies of $k(\pp)$ by Matlis \cite{Mat58}. Put $X = \bigcup_{M \in \Ccal}\Ass M$, then clearly $\Ccal \subseteq \Ccal_X$. To prove the converse, let $N \in \Mod R$ such that $\Ass N \subseteq X$. Then the injective envelope $E(N)$ of $N$ is a coproduct of copies of $E(k(\pp))$ with $\pp \in \Ass N$, again by \cite{Mat58}. Then $E(N) \in \Ccal$, so that also $N \in \Ccal$.
\end{proof}

The following Proposition is the key consequence of \cref{L:CX} we need in our classification theorem.

\begin{prop} \label{closurelimitCX} Let $(\Ccal, \Wcal)$ be a hereditary cotorsion pair cogenerated by a class  of pure-injective modules of finite injective dimension. Set $X=\Ass{\Ccal}$. Then $\Sub \Ccal=\{M \in \Mod R \mid \Ass M \subseteq X\}$.
\end{prop}

\begin{proof} Since $\Ccal$ is closed under injective envelopes  by \cref{C:C-inj-env}, we have $\Sub \Ccal = \Sub {\Ccal \cap \Ical_0}$, where $\Ical_0$ is the class of injective modules. It is straightforward to check that $\Sub \Ccal$ is closed under extensions (see \cref{eq:p-torsion-seq}). We claim that it is also closed under direct limits.  To this aim,  by \cite[Lemma 2.14]{GT12}, it is enough to show that it is closed under direct limits of well-ordered direct systems. By \cite[Lemma 3.6]{Hrb20}, the direct limit of any well-ordered direct system of modules from $\Sub {\Ccal \cap \Ical_0}$ is a submodule of a direct limit of objects from the class $\Ccal \cap \Ical_0$ which is closed under direct limits. This shows that $\Sub \Ccal$ is also closed by direct limits. 

Now the statement of the Proposition follows from \cref{L:CX}.
\end{proof}

Finally, we are ready to state and prove our main theorem.
\begin{theorem}\label{T:Torpair-characterisation}
Let $R$ be a commutative noetherian ring. There is a bijective correspondence between the following collections.
\begin{itemize}
		\item[(i)] Hereditary $\Tor$-pairs $(\Ecal, \Ccal)^\top$ generated by classes of modules of finite flat dimension,
		\item[(ii)] Hereditary cotorsion pairs $(\Ccal, \Wcal)$ cogenerated by classes of pure-injective modules of finite injective dimension,
		\item[(iii)] Functions $\phi\colon \Spec R \to \Zbb_{\geq 0}$ such that $\phi \leq \depth$,
		\item[(iv)] Sequences of subsets $X_0 \subseteq \dots \subseteq X_{n} \subseteq \dots \subset \Spec R$ such that $\Ass{\Omega_{-n}(R)} \subseteq X_n$.
	\end{itemize}
 The assignments are as follows:

    $(i) \rightarrow (ii)$: See \cref{tor-induce-cotor}.

    $(iii) \rightarrow (i):$ $\phi \mapsto (\Ecal_\phi,\Ccal_\phi)^\top$, where 
    $$\Ccal_\phi \coleq \{M \in \Mod R \mid \depth_{R_\pp}(M_\pp) \geq \phi(\pp) ~\forall \pp \in \Spec R\}.$$

    $(ii) \rightarrow (iii):$ $\Ccal \mapsto \phi_\Ccal$, where
    $$\phi_\Ccal(\pp) \coleq \inf \{i \in \Zbb_{\geq 0} \mid k(\pp) \in \Ccal_{(i)}\}.$$

    $(ii) \rightarrow (iv):$ $\Ccal \mapsto X_0 \subseteq \dots \subseteq X_{n} \subseteq \dots \subset \Spec R$, where
    $$X_n \coleq \Ass{\Ccal_{(n)}}.$$
\end{theorem}
\begin{proof}
The assignment $(i) \rightarrow (ii)$ of \cref{tor-induce-cotor} is clearly injective. The assignments $(iii) \rightarrow (i)$ and $(ii) \rightarrow (iii)$ are well-defined by \cref{function-to-torpair} and \cref{boundedbydepth}, respectively. 

There is a straightforward bijection $(iii) \rightarrow (iv)$ which assigns to a function $\phi$ subsets of $\Spec R$ of the form $X_n \coleq \bigcup_{0 \leq i \leq n}\phi^{-1}(i)$ for each $n \geq 0$, as by \cref{p:torpair-generator}, $\phi(\pp) \leq \depth(\pp)$ if and only if $\pp \in \Ass{\bigoplus_{i=0}^{\phi(\pp)}\Omega_{-i}(R)}$, and by construction $\phi(\pp)= \inf \{i \mid \pp \in X_i\}$. 
In this way, $(ii) \rightarrow (iv)$ is identified with $(ii) \rightarrow (iii)$ via \cref{l:associated-residue}.

The composition $(iii) \rightarrow (i) \xhookrightarrow{} (ii) \rightarrow (iii)$ is an identity by \cref{function-to-torpair}. To show that all the involved assignments are bijections, it remains to show that $(ii) \rightarrow (iii)$ is injective. Let $(\Ccal,\Wcal)$ be a cotorsion pair of the collection $(ii)$ and $\phi$ the function $\Spec R \to \Zbb_{\geq 0}$ associated via $(ii) \rightarrow (iii)$. We shall show that $\Ccal = \Ccal_\phi$. 

Using \cref{c:loc-cotor-pair}, we can reduce to the case of $R$ local. Then $\Ical = \Ical_{\dim(R)}$, which implies $\Ccal_{(\dim(R))}=\Mod R$.
For each $i \geq 0$, let $X_i = \Ass{\Ccal_{(i)}}$. Since, by the remarks before \cref{L:cotorsion-pair-shifts}, $(\Ccal_{(i)}, \Wcal_{(i)})$ is a cotorsion pair cogenerated by  pure-injective modules of finite injective dimension we can deduce,  from \cref{closurelimitCX},  that $\Sub \Ccal_{(i)}=\{M \in \Mod R \mid \Ass M \subseteq X_i\}$. 
By \cref{C-induction-step}, we get the relation $M \in \Ccal_{(i)}$ if and only if $\Ass{M} \subseteq X_i$ and $\Omega_{-1}(M) \in \Ccal_{(i+1)}$. Since $\Ccal_{(\dim(R))} = \Mod R$, we get by induction that $M \in \Ccal$ if and only if $\Ass {\Omega_{-i}(M)} \subseteq X_i$ for all $i \geq 0$. Then $\Ccal = \Ccal_\phi$ by the equivalence of $(ii)$ and $(iii)$ in \cref{p:torpair-generator}: $M \in \Ccal_\phi$ if and only if $\depth (M_\pp) \geq \inf \{i \in \Zbb_{\geq 0} \mid k(\pp) \in \Ccal_{(i)}\}$ if and only if $\pp \notin \Ass{\Omega_{-i}M}$ when $k(\pp) \notin \Ccal_{(i)}$ if and only if $\Ass{\Omega_{-i}M} \subseteq \Ass{\Ccal_{(i)}}$ for every $i$ if and only if $M \in \Ccal$.
\end{proof}

    \begin{rmk}
    It follows  that the character duals of the modules in $\Scal_\phi$ defined in \cref{function-to-torpair} are cogenerators of the cotorsion pair $(\Ccal_\phi,\Wcal_\phi)$ in Theorem~\ref{T:Torpair-characterisation}. However, it also follows  that all hereditary cotorsion pairs cogenerated by pure-injective modules of finite injective dimension can be cogenerated instead by modules which are better understood: Matlis duals. 
    
     By Proposition~\ref{p:torpair-generator}, $\{S_{\phi(\pp)}(\pp R_\pp; R_\pp)^{\vee_\pp} \mid S_{\phi(\pp)}(\pp R_\pp; R_\pp) \in \Scal_\phi \}$ also cogenerate the desired cotorsion pair, where $$S_{\phi(\pp)}(\pp R_\pp; R_\pp)^{\vee_\pp}=\mathrm{Hom}_{R_\pp} (S_{\phi(\pp)}(\pp R_\pp; R_\pp), E(R_{\pp}/\pp R_{\pp})).$$ 
    \end{rmk}

 \medskip

The following is a restriction of our main theorem to cotilting $\Tor$-pairs. This provides a generalisation of the classification of n-cotilting classes of \cite{AHPST14} to the case when the corresponding resolving subcategory has no uniform bound on projective dimension, which is relevant in the case $\dim(R) = \infty$, see \cref{rmk:cotilting-Tor-pairs}. Note that the bijection $(iii) \leftrightarrow (v)$ is precisely \cite[Theorem 1.2]{DT15} of Dao and Takahashi.

\begin{rmk} \label{grade}
Here, we call a function $\phi\colon \Spec R \to \Zbb_{\geq 0}$ \newterm{order-preserving} if $\phi(\qq) \leq \phi(\pp)$ whenever $\qq \subseteq \pp$. The \newterm{grade function} is the function $\grade: \Spec R \to \Zbb_{\geq 0}$ defined by the rule $\pp \mapsto \grade_R(\pp,R) = \inf \{i \in \Zbb_{\geq 0} \mid \Ext_R^i(R/\pp,R)\neq 0\}$, cf. \S \ref{ss:Koszul}. The relation with the $\depth$ function used in Theorem~\ref{T:Torpair-characterisation} is given by $\grade(\pp) = \inf \{\depth(\qq) \mid \qq \in V(\pp)\}$, see \cite[Proposition 1.2.10]{BH98}. 
In particular, $\grade$ is an order-preserving function such that $\grade \leq \depth$. Moreover, we deduce that, for any order-preserving function $\phi: \Spec R \to \Zbb_{\geq 0}$,  $\phi \leq \depth$ if and only if $\phi \leq \grade$. \end{rmk}

\begin{cor}\label{T:definable}
    Let $R$ be a commutative noetherian ring. The correspondence of \cref{T:Torpair-characterisation} restricts to another bijective correspondence between the following collections.
    \begin{itemize}
		\item[(i)] Hereditary $\Tor$-pairs $(\Ecal, \Ccal)^\top$ generated by classes of finitely generated modules of finite projective dimension,
		\item[(ii)]Cotilting $\Tor$-pairs $(\Ecal, \Ccal)^\top$,
		\item[(iii)] Order-preserving functions $\phi\colon \Spec R \to \Zbb_{\geq 0}$ such that $\phi \leq \depth$,
		\item[(iv)] Sequences $X_0 \subseteq \dots \subseteq X_{n} \subseteq \dots \subset \Spec R$ of generalisation closed subsets such that $\Ass{\Omega_{-i}(R)} \subseteq X_i$,
        \item[(v)] Resolving subcategories of $\mod R$ consisting of modules of finite projective dimension.
	\end{itemize}
\end{cor}
\begin{proof}
    The equivalence between $(iii)$ and $(iv)$ follows directly by restricting the assignment $\phi \mapsto X_n \coleq \bigcup_{0 \leq i \leq n}\phi^{-1}(i)$ defined in \cref{T:Torpair-characterisation}. It is straightforward to check that $\phi$ is order-preserving if and only if each $X_n$ is generalisation closed for each $n \geq 0$.
    
    To a function $\phi\colon \Spec R \to \Zbb_{\geq 0}$ as in $(iii)$, we assign the set 
    $$\Scal'_\phi = \{S_{\phi(\pp)}(\pp;R) \mid \pp \in \Spec R\},$$ 
    note that $\Scal'_\phi$ consists of finitely generated modules of finite projective dimension (see \cref{rem:koszul_cos}), because $\phi \leq \grade$ by Remark~\ref{grade}. Let $(\Ecal,\Ccal)^\top$ be the hereditary $\Tor$-pair such that $\Ccal = (\Scal'_\phi)^{\top}$. We have 
    $$\Ccal = \{M \in \Mod R \mid \grade_R(\pp,M) \geq \phi(\pp), \pp \in \Spec R\}$$ 
    by \cite[Proposition 5.12]{HS20}. We claim that this coincides with the $\Tor$-pair associated to $\phi$ via \cref{T:Torpair-characterisation}. Indeed, similarly to the discussion above, this follows directly from \cite[Proposition 1.2.10]{BH98} and the order-preservation of $\phi$.
    Since $\Scal'_\phi$ consists of finitely generated modules, $\Ccal$ is definable, and so $(\Ecal,\Ccal)^{\top}$ is a cotilting $\Tor$-pair. The assignment from $(iii)$ to $(v)$ is given by taking the smallest resolving subcategory of $\mod R$ which contains $\Scal'_\phi$.

    If $\dim(R)<\infty$ then we can see that these assignments are bijections directly by \cref{cotilting_characterisation}. It remains to explain the case of infinite Krull dimension. Let $(\Ecal,\Ccal)^{\top}$ be a cotilting $\Tor$-pair. For each maximal ideal $\mm$, $\Ccal_\mm$ is a cotilting class in $\Mod{R_\mm}$ by \cref{C:R_p-cotilting}. Let $\phi_\mm\colon \Spec {R_\mm} \to \Zbb_{\geq 0}$ be the associated order-preserving function to each $(\Ecal_\mm,\Ccal_\mm)^{\top}$ and let $\phi\colon \Spec R \to \Zbb_{\geq 0}$ be the function associated to $(\Ecal,\Ccal)^{\top}$. Since $\Omega_{-i}k(\pp) \in \Ccal$ if and only if $\Omega_{-i}k(\pp) \in \Ccal_\mm$ for each maximal ideal $\pp \subseteq \mm$, we see that $\phi_\mm(\pp R_\mm) = \phi(\pp)$ for all inclusions $\pp \subseteq \mm$. Then it is straightforward to show that $\Ccal = (\Scal'_\phi)^{\top}$. This establishes the bijection between items $(i),(ii),(iii)$. Finally, items $(i)$ and $(v)$ are in bijection by an argument similar to \cite[Theorem 13.49]{GT12}.
\end{proof}
It is natural to ask to what extent is our \cref{T:Torpair-characterisation} broader than the (essentially) known \cref{T:definable}. In view of the parametrising functions, this is to ask how strong of an assumption it is for a function $\phi\colon \Spec R \to \Zbb_{\geq 0}$ bounded by $\depth$ to be order-preserving. Using a classical result of Bass, we get the following dichotomy.
\begin{prop}\label{p:dim2}
    Let $R$ be a commutative noetherian ring. Then the following are equivalent:
    \begin{enumerate}
        \item[(i)] Every function $\phi\colon \Spec R \to \Zbb_{\geq 0}$ with $\phi \leq \depth$ is order-preserving.
        \item[(ii)] $\dim(R) \leq 1$.
    \end{enumerate}
\end{prop}
\begin{proof}
    $(i) \implies (ii)\colon$ Towards contradiction, assume that $\dim(R)>1$. Let $\qq \in \Spec R$ be a prime of height 2. Then a theorem of Bass \cite[Proposition 5.2, Corollary 5.3]{B62} provides a non-maximal prime ideal $\pp \subseteq \qq$ such that $\depth(\pp)= 1$. Then the function $\phi$ assigning 1 to $\pp$ and 0 to all other primes in $\Spec R$ is bounded by depth and not order-preserving.

    $(ii) \implies (i)\colon$ We have $\phi(\pp) \leq \depth(\pp) \leq \height(\pp)$ for every $\pp \in \Spec R$. Then $\phi(\pp) = 0$ if $\pp$ is a minimal prime and $\phi(\pp) \in \{0,1\}$ if $\pp$ is a maximal prime. Then, clearly, $\phi$ is order-preserving.
\end{proof}
\begin{rmk}
The $\depth$ function is the largest choice of a parametrising function in \cref{T:Torpair-characterisation}, and as such corresponds to the largest $\Tor$-pair $(\Ecal_\depth,\Ccal_\depth)^\top$ in the lattice which is subject to our classification. In another words, the $\Tor$-pair $(\Ecal_\depth,\Ccal_\depth)^\top$ is generated by the whole class $\Fcal$ consisting of \textit{all} modules of finite flat dimension. In the terminology we will introduce in \cref{s:rfd}, $\Ccal_\depth$ consists precisely of modules of restricted flat dimension zero. If $R$ is Cohen--Macaulay then $\depth = \grade =\height$ is itself an order-preserving function and thus the induced $\Tor$-pair is cotilting. In general though, the $\depth$ function may fail to be order-preserving. In fact, $\depth$ is order-preserving if and only if $\depth=\grade$ if and only if $R$ is almost Cohen--Macaulay, see \cref{s:rfd} again.

For example, consider a noetherian local ring $R$ with maximal ideal $\mm$ such that $\depth (\mm) = 0$. Then the only function associated to a cotilting $\Tor$-pair by \cref{T:definable} is the constant zero. However, if $\dim(R)>1$ then \cref{p:dim2} yields more depth-bounded functions and thus more hereditary $\Tor$-pairs subject to \cref{T:Torpair-characterisation}, including the one corresponding to the $\depth$ function itself.
\end{rmk}
\section{The regular case}\label{hered-tor-pair-reg-ring}
Let $R$ be a regular ring, that is, $R$ is a commutative noetherian ring such that every localisation of $R$ at a prime ideal has finite global dimension. We show that \cref{T:Torpair-characterisation} characterises {\it all} hereditary $\Tor$-pairs over $R$, and in fact, all hereditary cotorsion pairs cogenerated by pure-injective modules.  This follows readily if $R$ is local, or even when it has finite Krull dimension, but to obtain the general case we need to make the following observation first. Let $\mSpec R$ denote the set of maximal ideals of $R$.

\begin{prop}\label{coloc-orth}
	 Let $R$ be a commutative noetherian ring. Any hereditary cotorsion pair cogenerated by a class $\Qcal$ of pure-injective $R$-modules is cogenerated by the class $\bigcup_{\mm \in \mSpec R}\Qcal^\mm$, where $\Qcal^\mm \coloneqq \{N^\mm \mid N \in \Qcal\}$ for each maximal ideal $\mm$.
\end{prop}
\begin{proof}
	The statement follows directly from \cref{coloc-orth-pi}.
 \end{proof}

\begin{cor}\label{fininj-reg}
    Let $R$ be a regular ring. Any hereditary cotorsion pair cogenerated by pure-injective modules is cogenerated by pure-injective modules of finite injective dimension.
\end{cor}
\begin{proof}
    This follows directly from \cref{coloc-orth} as $N^\mm$ is a pure-injective $R$-module whose injective dimension over both $R_\mm$ and $R$ cannot exceed $\dim(R_\mm)$.
\end{proof}
Note that if $R$ is regular, then every hereditary $\Tor$-pair is generated by modules of finite flat dimension. Indeed, if a hereditary $\Tor$-pair $(\Ecal,\Ccal)^\top$ is generated by a class $\Gcal$ of modules, then it is also generated by the class $\bigcup_{\mm \in \mSpec R}\Gcal_\mm$, where $\Gcal_\mm = \{G_\mm \mid G \in \Gcal\}$. Unlike \cref{coloc-orth}, this follows easily by noting that $\Tor_i^R(M,G) = 0$ if and only if $0 = \Tor_i^R(M,G)_\mm \cong \Tor_i^{R_\mm}(M,G_\mm)$ for each $\mm \in \mSpec R$.

In the regular case, our \cref{T:Torpair-characterisation} now allows the following formulation. Recall that over a regular (or more generally, Cohen--Macaulay) ring $R$, we have $\height (\pp) = \grade(\pp) = \depth(\pp)$ for all $\pp \in \Spec R$.  
\begin{theorem}\label{T:Torpair-characterisation-regular}
Let $R$ be a regular commutative noetherian ring. There is a bijective correspondence between the following collections.
\begin{itemize}
		\item[(i)] Hereditary $\Tor$-pairs $(\Ecal, \Ccal)^\top$,
		\item[(ii)] Hereditary cotorsion pairs $(\Ccal, \Wcal)$ cogenerated by classes of pure-injective modules,
		\item[(iii)] Functions $\phi\colon \Spec R \to \Zbb_{\geq 0}$ such that $\phi \leq \depth$,
		\item[(iv)] Sequences $X_0 \subseteq \dots \subseteq X_{n} \subseteq \dots \subseteq \Spec R$ such that $\Ass{\Omega_{-i}(R)} \subseteq X_i$.
	\end{itemize}
\end{theorem}

\begin{rmk}
    Note that, even in the regular case above, the theory of pure-injective modules can be far from transparent. In fact, outside of the hereditary case of dimension at most one, such rings always admit a superdecomposable pure-injective module, see \cite[Summary 8.71]{JL89}.
\end{rmk}

The collection $(i)$ of \cref{T:Torpair-characterisation-regular} has an obvious symmetry: Given a hereditary $\Tor$-pair $(\Ecal, \Ccal)^\top$, also $(\Ccal, \Ecal)^\top$ is a hereditary $\Tor$-pair. This is reflected by the corresponding functions in the following way. 
\begin{prop}\label{reg-ring-other-function}
If $\phi$ is the function $\Spec R \to \Zbb_{\geq 0}$ associated to $(\Ecal, \Ccal)^\top$ in \cref{T:Torpair-characterisation-regular}, then we claim that the function associated to $(\Ccal, \Ecal)^\top$ is exactly $\psi(\pp) \coleq \height (\pp) - \phi(\pp)$.
\end{prop}
\begin{proof}
To see the claim,  note first that as $\Ecal$ is hereditary, so is $\Ccal$, and by construction $S_{\phi(\pp)}(\pp R_\pp; R_\pp) \in \Ecal$. The Koszul complex $K_\bullet (\pp R_\pp; R_\pp)$ is a projective resolution of $k(\pp)$ as $R_\pp$ is a regular ring by assumption, so $S_{j}(\pp R_\pp; R_\pp)$ is a $(\height (\pp)-j)$-syzygy of $k(\pp)$ for every $0 \leq j \leq \height( \pp)$. In particular, $\depth_{R_\pp} S_{j}(\pp R_\pp; R_\pp) = \height(\pp)-j$, so it follows from the characterisation in \cref{p:torpair-generator} that $S_{j}(\pp R_\pp; R_\pp) \in \Ccal$ for $j \leq \height(\pp)-\phi(\pp)$. So the associated function to the hereditary bounded $\Tor$-pair $(\Ccal, \Ecal)^\top$ is exactly $\height(\pp)-\phi(\pp)$, as claimed. 
\end{proof}
\begin{rmk}\label{opposite-tor}
    Let $R$ be a commutative noetherian ring, and let $(\Ecal,\Ccal)^\top$ be a hereditary $\Tor$-pair generated by a class of modules of finite flat dimension. If $R$ is singular then the opposite $\Tor$-pair $(\Ccal,\Ecal)^\top$ is never subject to the classification \cref{T:Torpair-characterisation}. Indeed, towards contradiction assume that both $(\Ccal,\Ecal)^\top$ and $(\Ecal,\Ccal)^{\top}$ are generated by modules of finite flat dimension. Passing to a localisation, we can assume that $R$ is a local singular ring. Then $\Ecal,\Ccal \subseteq \Fcal_d$, where $d = \dim R$. This is a contradiction, as the singularity ensures that there is a module $M$ such that $M \in \Fcal^{\top}$ but $M \not\in \Fcal$. Indeed, we can choose $M$ as a cocycle module in an acyclic, but not contractible, complex of projective $R$-modules (see \cite[Theorem 0]{CET25}). 
    
    For another example, let $\pp$ be a prime ideal of $R$. Take the syzygy $\Omega_d(k(\pp))$ where $d= \depth(R_\pp)$. Then by Ext-Tor duality and \cref{Torffddepth}, $\Omega_d(k(\pp))\in \Fcal^\top$ and $k(\pp)$ is of finite flat dimension in $\Mod R$ if and only if $R_\pp$ is regular.
\end{rmk}
\begin{prop}\label{both-definable}
    Let $R$ be a regular commutative noetherian ring. The bijection of \cref{T:Torpair-characterisation-regular} restricts to a bijection between:
    \begin{enumerate}
        \item[(i)] $\Tor$-pairs $(\Ecal,\Ccal)^\top$ such that both $\Ecal$ and $\Ccal$ are definable,
        \item[(ii)] Functions $\phi\colon \Spec R \to \Zbb_{\geq 0}$ satisfying:
        \begin{itemize}
            \item $\phi(\qq)=0$ for any minimal prime $\qq$, and
            \item $0 \leq \phi(\pp)-\phi(\qq) \leq 1$ for all immediate inclusions $\qq \subsetneq \pp$.
        \end{itemize}
    \end{enumerate}
\end{prop}
\begin{proof}
    First, notice that any function $\phi$ as in $(ii)$ satisfies $0 \leq \phi(\pp) \leq \height(\pp)$, and so corresponds to a hereditary $\Tor$-pair via \cref{T:Torpair-characterisation-regular}. Indeed, given any $\pp \in \Spec R$, consider a chain of immediate inclusions $\qq = \qq_0 \subsetneq \qq_1 \subsetneq \cdots \subsetneq \qq_n = \pp$, where $\qq$ is some minimal prime ideal. Then the assumption on $\phi$ yields $0 = \phi(\qq) \leq \phi(\pp) \leq n - \phi(\qq) = n$, and $n = \height(\pp)$ as $R$ is catenary.

    In view of \cref{T:definable} and \cref{reg-ring-other-function}, it remains to check that the functions of $(ii)$ are precisely those functions $\phi\colon \Spec R \to \Zbb_{\geq 0}$ of \cref{T:Torpair-characterisation-regular} such that both $\phi$ and $\psi = \height - \phi$ are order-preserving. Assume that $\phi$ is order-preserving. Then $\psi$ is order-preserving precisely if for each immediate inclusion $\qq \subsetneq \pp$ of primes we have $\psi(\pp) - \psi(\qq) \geq 0$. We have
    $$\psi(\pp) - \psi(\qq) = (\height(\pp)-\height(\qq)) - (\phi(\pp)-\phi(\qq)) = 1 - (\phi(\pp)-\phi(\qq)),$$
    which shows that indeed $\psi(\pp) - \psi(\qq) \geq 0$ if and only if $\phi(\pp)-\phi(\qq) \leq 1$.
\end{proof}
\begin{cor}{\cite[Theorem 16.31]{GT12}}
    Let $R$ be a regular commutative noetherian ring of Krull dimension 1 (or, equivalently, a hereditary commutative noetherian ring). Then given any $\Tor$-pair $(\Ecal,\Ccal)^\top$, both $\Ecal$ and $\Ccal$ are definable, and thus 1-cotilting classes.
\end{cor}
\begin{proof}
    Let $\phi\colon \Spec R \to \Zbb_{\geq 0}$ be a function satisfying $0 \leq \phi(\pp) \leq \height(\pp)$. Then the function is determined by values on maximal ideals $\mm$, with the only possible values being 0 or 1. Such a function automatically satisfies the condition $(ii)$ of \cref{both-definable}.
\end{proof}

\section{The restricted flat dimensions}\label{s:rfd}
Let $R$ be a commutative noetherian ring. Christensen, Foxby and Frankild introduced the restricted homological dimensions over a commutative noetherian ring in various works, see \cite{CFF02} and \cite[Chapter 5]{C00}. Instead of considering a $\Tor$-pair, their approach involved introducing the \newterm{large and small restricted flat dimensions} of a module, denoted $\Rfd$ and $\rfd$ respectively (also called the large and small restricted $\Tor$-dimensions). They are defined as follows for an $R$-module $M$ over a commutative noetherian ring $R$. 
\[
\Rfd_R(M)\coleq \sup \{i \geq 0  \mid \Tor^R_i(F, M) \neq 0 \text{ for some } F \in \Fcal\}
\]
\[
\rfd_R(M)\coleq \sup \{i \geq 0  \mid \Tor^R_i(F, M) \neq 0 \text{ for some } F \in \Pcal^{<\aleph_0}\}
\]

We will let $\Rcal\Fcal_n$ denote the modules of large restricted flat dimension at most $n$ and also let $(\Fcal_{(n)}, \Rcal\Fcal_n)^\top$ be the hereditary $\Tor$-pair, which is generated by $n$-th yokes of modules belonging to $\Fcal$. Note that this $\Tor$-pair corresponds to the function $\pp \mapsto \sup\{0,\depth(\pp)-n\}$ on $\Spec R$ via \cref{T:Torpair-characterisation}, see \cref{function-to-torpair}.

In \cite{CFF02}, the authors succeed in providing a kind of analogue of the Auslander-Buchsbaum formula for the restricted flat dimensions, see \cite[Theorem 2.4(b)]{CFF02}, which we recover in our main \cref{T:Torpair-characterisation} by setting $\phi\coleq\depth$. 

It is straightforward to see that $\rfd_R(M) \leq \Rfd_R(M) \leq \fd_R(M)$ for any $R$-module $M$. However, to determine when each of these inequalities is an equality
is less straightforward, see \cite{CFF02} for more details. In particular, they show that if $R$ is a local ring, then $\rfd_R(M) = \Rfd_R(M)$ for every (finitely generated) $R$-module $M$ if and only if $R$ is \newterm{almost Cohen--Macaulay}, that is if $\dim R_\pp-\depth(R_\pp) \leq 1$ for every prime ideal $\pp$, see \cite[Theorem 3.2]{CFF02}. Note that $R$ is almost Cohen--Macaulay if and only if $\grade = \depth$ by \cite[Lemma 3.1]{CFF02}.

In the next proposition, we will extend this characterisation to all noetherian rings by combining the original results from \cite{CFF02} with some newer results from \cite{HLG24} together with \cref{def-tor-pairs} (which is \cite[Theorem 2.3]{AHT04}). The equivalence of $(iii)$ and $(iv)$ is generalised in our \cref{T:definable}.

\begin{prop}
    Let $R$ be a commutative noetherian ring. Then the following are equivalent: \begin{enumerate}
        \item[(i)] $R$ is almost Cohen--Macaulay,
        \item[(ii)] $\rfd_R(M) = \Rfd_R(M)$ for every (finitely generated) $R$-module $M$,
        \item[(iii)] $\Rcal \Fcal _0$ is a definable class, that is, $ \Rcal \Fcal_0$ forms the right-hand class of a cotilting $\Tor$-pair generated by $\Pcal^{<\aleph_0}(R)$,
        \item[(iv)] $\Fcal(R) \subseteq \varinjlim \Pcal^{<\aleph_0}(R)$, where the equality holds if and only if $R$ is of finite dimension,
        \item[(v)] $\Fcal_n(R) = \varinjlim \Pcal^{<\aleph_0}_n(R)$ for every $n>0$,
        \item[(vi)] $\Fcal_n(R_\pp) = \varinjlim \Pcal^{<\aleph_0}_n(R_\pp)$ for every $n>0$ for every $\pp \in \Spec R$.
    \end{enumerate}
\end{prop}
\begin{proof}
    $(i)\Leftrightarrow (v) \Leftrightarrow (vi)$ is by \cite[Theorem B]{HLG24} and as being almost Cohen--Macaulay is a local condition by definition.
    
    $(i)\Rightarrow (ii)$: This holds by \cite[Corollary 3.3]{CFF02}. Alternatively, the equivalence $(i)\Leftrightarrow (ii)$ follows from \cref{T:Torpair-characterisation}, \cref{T:definable}, and the above mentioned fact that $(i)$ amounts to $\grade = \depth$. Indeed, $\rfd_R = \Rfd_R$ is equivalent to $\Rcal \Fcal_n$ being definable by \cref{T:definable}. By the same result, this amounts to $\depth$ being order-preserving, and thus equal to $\grade$.

    $(ii)\Rightarrow (iii)$: By \cite[Theorem 3.2]{CFF02} we can suppose that $\rfd_R(M) = \Rfd_R(M)$ for every $R$-module $M$. In particular, $\Rcal \Fcal _0$ is closed under products as it is generated as a $\Tor$-orthogonal to a set of finitely presented modules, and such classes are closed under products.

    $(iii)\Rightarrow (vi)$
    It follows that for every prime ideal $\pp$, $\Rcal\Fcal_0(R_\pp) = \Rcal\Fcal_0(R)\cap \Mod R_\pp$ as $\Tor$ commutes with localisations, so $\Rcal\Fcal_0(R_\pp)$  is a definable class for each prime $\pp \in \Spec R$. By \cref{def-tor-pairs}, we conclude that $\Fcal_n(R_\pp) = \varinjlim \Pcal_n^{<\aleph_0}(R_\pp)$ for every prime $\pp$ where $n\coleq \dim R_\pp$. To see that the statement holds for every $n$, apply \cite[Theorem B]{HLG24}.  

    $(v) \Rightarrow (iv)$: This is clear.
    
    $(iv) \Rightarrow (vi)$: Take $F \in \Fcal_n(R_\pp)\subseteq \Fcal_n(R)$ where $n \coleq \dim R_\pp$. By assumption, $F = \varinjlim M_i$, where each $M_i$ is a finitely presented $R$-module of finite projective dimension. As $F$ is an $R_\pp$-module by assumption, and direct limits commute with localisation, $F \cong F\otimes_RR_\pp = \varinjlim (M_i\otimes_RR_\pp)$, where the right-hand side is a direct limit in $\Mod R_\pp$ of finitely presented $R_\pp$-modules of projective dimension at most $n$. That the statement holds for all $n$ then follows by \cite[Theorem B and Subsection 2.13]{HLG24}. 
\end{proof}
For the weaker condition that $\Fcal_n(R) = \varinjlim \Pcal^{<\aleph_0}_n(R)$ for a particular $n>0$, see \cite[Theorem B]{HLG24}.

\begin{rmk} 
The modules with vanishing restricted flat dimension were also introduced in \cite{Xu96} with the nomenclature \newterm{strongly torsion-free} modules over an arbitrary ring. Additionally, we recover Xu's characterisation of the strongly torsion-free modules over Gorenstein rings in our \cref{T:Torpair-characterisation}, see \cite[Theorem 5.4.8]{Xu96}.
\end{rmk} 

\bibliographystyle{amsalpha}
\bibliography{bibitems}
\end{document}